\documentclass[twoside,11pt,reqno]{amsart}
\usepackage{amscd, bbm}
\usepackage{amssymb,mathabx}
\usepackage{todonotes}
\usepackage{mathrsfs,wasysym, amscd}
\usepackage[cmtip,arrow,all]{xy}
\usepackage{pb-diagram,pb-xy}\usetikzlibrary{matrix, arrows}

\makeatletter

 \usepackage{geometry}
\@addtoreset{equation}{section}

\newcommand{\arxiv}[1]{{\tt arXiv:#1}}
\newtheorem{Theorem}{Theorem}[section]
\newtheorem{Lemma}[Theorem]{Lemma}

\newtheorem{Corollary}[Theorem]{Corollary}

\theoremstyle{remark}
\newtheorem{Remark}[Theorem]{Remark}

\newtheorem{Example}[Theorem]{Example}

\def\half{{\textstyle\frac{1}{2}}}

\def\HC{{\operatorname{HC}}}

\def\Lie{\operatorname{Lie}}

\def\gr{\operatorname{gr}}

\def\ad{{\operatorname{ad}}}

\def\K{{\mathbb K}}
\def\C{{\mathbb C}}
\def\F{{\mathbb F}}
\def\Z{{\mathbb Z}}
\def\Zg{\N}

\def\k{{\mathbbm k}}

\def\N{{\mathbb N}}

\def\F{{\mathrm F}}

\def\tD{{\widetilde D}}
\def\0{{\bar 0}}
\def\1{{\bar 1}}

\def\eps{{\varepsilon}}
\newcommand{\g}{\mathfrak{g}}

\newcommand{\gl}{\mathfrak{gl}}
\renewcommand{\sl}{\mathfrak{sl}}
\newcommand{\z}{\mathfrak{z}}

\renewcommand{\S}{\mathfrak{S}}

\def\sl{\mathfrak{sl}}

\def\g{{\mathfrak g}}
\def\gl{\mathfrak{gl}}

\def\r{\mathfrak r}

\def\S{\mathfrak{S}}

\def\ds{{\dot s}}
\def\dsigma{{\dot \sigma}}

\newcommand{\gs}{{\g_{\sigma}}}
\newcommand{\Us}{U(\g_{\sigma})}
\newcommand{\Ss}{S(\g_{\sigma})}
\newcommand{\Zhcs}{Z_\HC(\g_{\sigma})}
\newcommand{\Zps}{Z_p(\g_{\sigma})}
\newcommand{\Zs}{Z(\g_{\sigma})}
\newcommand{\Ys}{Y_n(\sigma)}

\newcommand{\sE}{{^\sigma\!} E}
\newcommand{\dsE}{{^\dsigma\!} E}
\newcommand{\sF}{{^\sigma\!} F}
\newcommand{\dsF}{{^\dsigma\!} F}

\newcommand{\glnt}{\gl_n[t]}

\def\Lambdap{\Lambda\!^+\!}

\def\isoto{\overset{\sim}{\longrightarrow}}

\begin{document}
\title[Yangians and shifted Yangians]{\boldmath The $p$-centre of Yangians and shifted Yangians}
\author{Jonathan Brundan and Lewis Topley}
\thanks{Research of J.B. supported in part by NSF grant DMS-1700905.}

\begin{abstract}
We study the Yangian $Y_n$ associated to the general linear Lie
algebra $\gl_n$
over a field of {positive
  characteristic}, as well as its shifted analog $Y_n(\sigma)$. Our main result
gives a description of the centre of $Y_n(\sigma)$: it is a polynomial algebra
generated by its {\em Harish-Chandra centre} (which lifts the centre in characteristic zero)
together with a large {\em $p$-centre}. Moreover,
$Y_n(\sigma)$ is free as a module over its center.
In future work, it will be seen that every reduced enveloping algebra
$U_\chi(\mathfrak{gl}_n)$ 
is Morita equivalent to a quotient of an appropriate choice of
shifted Yangian, and so our results 
will have applications in classical representation theory.
\end{abstract}

\maketitle

\section{Introduction}

The Yangian $Y_n$ associated to the Lie algebra
$\mathfrak{gl}_n$ over the complex numbers was introduced by the
St. Petersburg school 
around 1980, and is one of the fundamental examples 
of quantum groups which can be defined using the {\em RTT formalism} of Faddeev,
Reshetikhin and Takhtadzhyan \cite{FRT}.
In \cite{D}, Drinfeld introduced another presentation, allowing him to extend
the notion of Yangian to all semisimple Lie algebras.
In this paper, we initiate a study of the Yangian of $\gl_n$, 
and its shifted analog in the sense of \cite{BK2},
over an arbitrary field $\k$ of positive characteristic $p > 0$.

Let us explain our motivation for doing this.
In characteristic zero, the results of \cite{BK2} show that
the shifted Yangians have some truncations which
are isomorphic to the {\em finite $W$-algebras} associated
to nilpotent orbits in $\gl_n$, as defined for
example by Premet \cite{P}. Finite $W$-algebras had appeared earlier,
in work of Kostant and subsequent work by mathematical physicists, but Premet's
motivation came from the representation theory of Lie algebras in
positive characteristic:
building on his work on the Kac-Weisfeiler conjecture \cite{P0},
he also discovered some remarkable 
finite-dimensional
{\em reduced finite $W$-algebras}, which are Morita equivalent to reduced enveloping algebras
of modular reductive Lie algebras.

Recently, the second
author jointly with Goodwin \cite{GT} has introduced 
{\em modular finite $W$-algebras}, which are the precise analog of
finite $W$-algebras in characteristic $p$. For sufficiently large $p$,
the same algebras appeared already in Premet's foundational work
\cite{PrCQ}, where they were constructed by reducing the complex finite
$W$-algebras modulo $p$.
Modular finite $W$-algebras have a large {\em $p$-centre}, and
Premet's reduced finite $W$-algebras may be recovered from modular
finite $W$-algebras by specializing at a $p$-central character.
As in characteristic zero, for the Lie algebra $\mathfrak{gl}_n$,
the modular finite $W$-algebra is a truncation of the {\em modular
  shifted Yangian} introduced here. 
The reduced finite
$W$-algebras for $\mathfrak{gl}_n$ can then be obtained from truncated
modular shifted Yangians by specializing at some $p$-central
character. For this to be useful in practice, one needs precise information about the
structure of the centre of the shifted Yangian in positive
characteristic.
This is the main purpose of the present article.

In the remainder of the introduction, we focus on the special case of
the Yangian $Y_n$, rather than its shifted analog $Y_n(\sigma)$, and
give a quick outline of our main results in this case.
We define the Yangian $Y_n$ over any field $\k$ 
to be the associative $\k$-algebra
with generators $\big\{T_{i,j}^{(r)}\:\big|\:1 \leq i,j \leq n, r > 0\big\}$
subject to the relations
\begin{equation}\label{rttnew}
\left[T_{i,j}^{(r+1)}, T_{k,l}^{(s)}\right] - \left[T_{i,j}^{(r)}, T_{k,l}^{(s+1)}\right] = T_{k,j}^{(r)} T_{i,l}^{(s)} - T_{k,j}^{(s)} T_{i,l}^{(r)}
\end{equation}
for all $1 \leq i,j,k,l \leq n$ and $r,s \geq 0$, 
adopting the convention that $T_{i,j}^{(0)} := \delta_{i,j}$.
When $\k$ is of characteristic zero, this is the usual RTT
presentation as found e.g. in \cite{MNO}, and in this 
case it is well known that $Y_n$ is 
a filtered deformation of the universal enveloping algebra
$U(\mathfrak{gl}_n[t])$ of the polynomial current algebra
$\gl_n[t]:=\gl_n \otimes \k[t]$. In fact, the same assertion is true when $\k$
is of positive characteristic, which the reader may regard as the first
evidence that this is a reasonable object to consider in the modular setting.

In characteristic zero, the centre $Z(Y_n)$ is generated freely by 
elements $\big\{C^{(r)}\:\big|\:r > 0\big\}$; see \cite[Theorem 2.13]{MNO}.
Moreover, $C^{(r+1)}$ is a lift of the obvious central
element $z_{r} := (e_{1,1}+\cdots+e_{n,n}) \otimes t^{r}$ in the
centre $Z(\gl_n[t])$ of the associated graded algebra $U(\gl_n[t])$.
The usual way to define these elements formally
is to work with generating functions in $Y_n[[u^{-1}]]$ following
\cite{MNO},
setting
\begin{equation}\label{amazing}
T_{i,j}(u) := \sum_{r \geq 0} T_{i,j}^{(r)} u^{-r} \in Y_n[[u^{-1}]].
\end{equation}
Then the central elements $C^{(r)}$ are the coefficients of
the {\em quantum determinant}
\begin{equation}\label{qdet}
C(u) = \sum_{r \geq 0} C^{(r)} u^{-r} := 
\sum_{g \in \S_n} \operatorname{sgn}(g)
T_{g(1),1}(u)
T_{g(2),2}(u-1)
\cdots T_{g(n),n}(u-n+1)
.
\end{equation}
These definitions 
makes sense when $\operatorname{char} \k > 0$ too, yielding
elements $\big\{C^{(r)}\:\big|\:r > 0\big\}$ which freely generate a
subalgebra $Z_\HC(Y_n)$ of $Z(Y_n)$, which we call the {\em Harish-Chandra centre}.
This corresponds at the level of the 
associated graded algebra to the subalgebra of 
$Z(\gl_n[t])$ generated by $\{z_{r}\:|\:r \geq 0\}$.

In positive characteristic, the current algebra $\gl_n[t]$ admits a natural structure of {\em
  restricted Lie algebra}, with $p$-map $x \mapsto x^{[p]}$ defined from
$(e_{i,j}t^r)^{[p]} := \delta_{i,j} e_{i,j} t^{rp}$.
Consequently, $U(\gl_n[t])$ also has a large {\em $p$-centre}
generated by the elements 
\begin{equation}\label{tolift}
\big\{(e_{i,j}t^{r})^p -
\delta_{i,j} e_{i,j}t^{rp}
\:\big|\:1 \leq i,j \leq n, r \geq 0\big\},
\end{equation}
and it 
is natural to look for lifts of these elements in $Z(Y_n)$.
We will show that such lifts are provided by the coefficients
$\big\{S_{i,j}^{(pr)}\:\big|\:1 \leq i,j \leq n, r > 0\big\}$ of the
power series
\begin{equation}
S_{i,j}(u) = \sum_{r \geq 0} S_{i,j}^{(r)} u^{-r} := T_{i,j}(u)
T_{i,j}(u-1)\cdots T_{i,j}(u-p+1).
\end{equation}
In
fact, these coefficients freely generate another subalgebra $Z_p(Y_n)$ of $Z(Y_n)$ which we call
the {\em $p$-centre} of $Y_n$.
Together, $Z_\HC(Y_n)$ and $Z_p(Y_n)$ generate the full centre
$Z(Y_n)$. 
In fact, we will show that $Z(Y_n)$ is a free polynomial algebra
generated by
\begin{equation}
\big\{C^{(r)}\:\big|\:r > 0\big\}
\cup\big\{S_{i,j}^{(pr)}\:\big|\:1 \leq i,j\leq n\text{ with }(i,j)\neq (1,1),r
> 0\big\}.
\end{equation}
Moreover, $Z_\HC(Y_n) \cap Z_p(Y_n)$ is
freely generated 
by elements $\big\{BC^{(pr)}\:\big|\:r > 0\big\}$ defined from
\begin{equation}
BC(u) = \sum_{r \geq 0} BC^{(r)} u^{-r} := C(u) C(u-1)\cdots C(u-p+1).
\end{equation}
Finally, $Y_n$ is free (of infinite rank) as a module over its centre.

To prove these statements, we found that it was easier to work
initially with the Drinfeld presentation rather than the RTT
presentation of $Y_n$, exploiting a different choice of lifts of 
the elements (\ref{tolift}) adapted to the Drinfeld generators.
There is also a similar family of lifts which plays the analogous role
for the shifted Yangians $Y_n(\sigma)$. We refer the
reader to Theorem~\ref{main} for the precise statement of our main result
in the general setup.

The remainder of the article is organized as follows.
In section 2, we make some auxiliary calculations with the sorts of
power series 
that will be used to define the central elements of $Y_n(\sigma)$
later on.
In particular, the results in this section are needed to compute
images of these central elements in the associated graded algebra.
Then section 3 investigates the centre of this 
associated graded algebra. The main result here, Theorem~\ref{loopcentre}, provides
the key upper bound needed in our computation of the centre of $Y_n(\sigma)$
later on.
We introduce the modular Yangians and shifted Yangians in section 4,
extending some of the fundamental results such as Drinfeld-type
presentation from characteristic zero to characteristic $p$.
Section 5 contains our main results describing the centre 
of $Y_n(\sigma)$ in positive
characteristic.
Finally, in section 6, we 
define a modular version of the Yangian of $\sl_n$,
which is a certain subalgebra $SY_n$ of
$Y_n$. We give a presentation for $SY_n$ valid in any characteristic;
see Theorem~\ref{newsy}.
As we explain in the subsequent remark, this presentation 
is equivalent to the usual 
Drinfeld presentation for the Yangian of $\sl_n$ taken over the field
$\k$ whenever $\operatorname{char} \k \neq 2$; Drinfeld's presentation does
not even make sense in characteristic 2 since it involves some halves.
We also show  that
$Y_n \cong SY_n \otimes Z_{\HC}(Y_n)$ providing $p \nmid n$, and use this
to deduce the results about $Z(Y_n)$ in terms of the RTT generators as
formulated in this introduction.

\vspace{2mm}
\noindent
{\bf General conventions.}
Let $\Zg$ denote the set of natural numbers $\{0,1,2,\dots\}$.
Always, $\k$ will be a ground field of characteristic $p > 0$
and $\otimes$ denotes $\otimes_{\k}$.
Let $\gl_n := \gl_n(\k)$ for some fixed $n \geq 1$.
We denote its matrix units by $\{e_{i,j}\:|\: i,j=1,\dots,n\}$ as usual.

By a {\em filtration} on a (unital) $\k$-algebra $A$, we always mean an ascending chain of
subspaces $$
\F_0 A \subseteq \F_1 A \subseteq \F_2 A
\subseteq \cdots
$$
with $A = \bigcup_{r\geq 0} \F_r A$ such that $1 \in \F_0 A$ and
$(\F_r A) (\F_s A) \subseteq \F_{r+s} A$ for all $r,s \in
\N$.
Setting $\F_{-1} A := \{0\}$ by convention, the {\em associated graded
 algebra} is 
the (unital) positively graded $\k$-algebra 
$\gr A = \bigoplus_{r \in \N} (\gr  A)_r$
with
$(\gr A)_r :=
\F_r A /
\F_{r-1} A.$
Let $\gr_r:\F_r A \rightarrow (\gr A)_r$ be the canonical quotient
map.
Given a subalgebra $B$ of $A$, we always give $B$ the induced
filtration defined from 
$F_r B := B \cap F_r A$. Then the associated graded algebra $\gr
B$ is naturally identified with a graded subalgebra of $\gr A$.

Given an $n$-tuple $\lambda =(\lambda_1,\dots,\lambda_n) \in \Zg^n$, 
let $\ell(\lambda) := \big|\{i=1,\dots,n\:|\: \lambda_i \neq 0\}\big|$ and 
$|\lambda| := \sum_{i=1}^{n} \lambda_i$.
Then set
 \begin{align*}
\Lambda(n,r)&:= \left\{ \lambda \in \Zg^n \:\big|\: |\lambda| = r\right\},&
\Lambdap(n,r) &:= \left\{\lambda \in \Lambda(n,r) \:\big|\: \lambda_1 \geq \lambda_2 \geq \cdots \geq \lambda_{n}\right\}.
\end{align*}
Elements of $\Lambda(n,r)$ and $\Lambdap(n,r)$ are $n$-part 
{\em  compositions} and {\em partitions} of $r$, respectively.
The symmetric group $\S_n$
acts on the left on $\Lambda(n,r)$ by place permutation, and
$\Lambdap(n,r)$ gives a set of orbit representatives.
Also let $\S_\lambda \leq \S_n$ denote the stabilizer of $\lambda \in
\Lambdap(n,r)$, which is a parabolic subgroup of $\S_n$.
Let $\subseteq$ be the partial order on $\N^n$ 
defined from $\lambda \subseteq \mu$ if $\lambda_i \leq \mu_i$ for each $i=1,\dots,n$.

We will use the following elementary facts several times.

\begin{Lemma}\label{L:zerolem}
For $\lambda \in \Lambdap(p,r)$,
$|\S_p/\S_{\lambda}|$ is
non-zero in the field $\k$ if and only if $p | r$ and
$\lambda_1=\dots=\lambda_p = r/p$, 
in which case it equals $1$.
\end{Lemma}
\begin{proof}
Obvious.
\end{proof}

\begin{Lemma}\label{hop}
In any associative $\k$-algebra $A$, we have that
$(\ad\, x)^p = \ad\,(x^p)$
for $x \in A$.
\end{Lemma}

\begin{proof}
We have that $\ad\, x = \lambda_x - \rho_x$,
where $\lambda_x,\rho_x:A \rightarrow A$ denote the commuting operations of left and right multiplication by
$x$, respectively. Hence $(\ad\, x)^p = (\lambda_x - \rho_x)^p =
(\lambda_x)^p - (\rho_x)^p = \lambda_{x^p}-\rho_{x^p} = \ad (x^p)$ 
since we are in characteristic $p$.
\end{proof}

\section{Power series}\label{S:powerseries}

We often need to manipulate power series in $A[[u^{-1}]]$,
where $A$ is some $\k$-algebra (typically, some Yangian). When doing
so, it is sometimes convenient to
work in the slightly larger ring $A[[u^{-1}]][u]$ of formal Laurent
series in $u^{-1}$ with coefficients in $A$.
In this section, we make some general observations about some
particular power
series which will be needed in several key places later on.

\subsection{Series of type I}\label{S:typeI}
In this subsection, we assume that $A_{\rm I}$ is 
a filtered $\k$-algebra  and that we are given elements
$\left\{X^{(r)} \in \F_{r-1} A_{\rm I}\:|\: r \geq 0\right\}$
such that $\left[\gr_r X^{(r+1)}, \gr_s X^{(s+1)}\right] = 0$ for all $r,s > 0$.
Since $\F_{-1} A_{\rm I} = \{0\}$ by our general conventions, we
necessarily have that $X^{(0)} = 0$. Consider the power series
\begin{equation}
X_{\rm I}(u) = \sum_{r \geq 0} X_{\rm I}^{(r)} u^{-r} := \left(\sum_{r \geq 0} X^{(r)} u^{-r}\right)^p \in A_{\rm I}[[u^{-1}]].
\end{equation}

\begin{Lemma}\label{I}
We have that $X_{\rm I}^{(r)} = 0$ for $r < p$
and $X_{\rm I}^{(p)} = \big(X^{(1)}\big)^p \in \F_0 A_{\rm I}$.
If $r > p$ and $p \mid r$ then $X_{\rm I}^{(r)} \in \F_{r-p} A_{\rm I}$ and 
$X_{\rm I}^{(r)} \equiv (X^{(r/p)})^p \pmod
{\F_{r-p-1} A_{\rm I}}$.
Finally, if $r > p$ and $p \nmid r$ then $X_{\rm I}^{(r)} \in \F_{r-p-1}
A_{\rm I}$.
\end{Lemma}

\begin{proof}
We obviously have 
that
$X_{\rm I}^{(r)} = \sum_{\lambda \in \Lambda(p,r)} X^{(\lambda)}$
where
$X^{(\lambda)} := 
X^{(\lambda_1)} X^{(\lambda_2)} \cdots X^{(\lambda_{p})}$.
Since $X^{(0)} = 0$, the result follows at once when $r < p$.
Now suppose that $r \ge p$. Then we get 
that
$X_{\rm I}^{(r)} \in \F_{r-p} A_{\rm I}$. The commuting assumption
on the elements $\gr_{r-1} X^{(r)}$ gives that
$$
X^{(g\cdot\lambda)} \equiv
X^{(\lambda)} \pmod{\F_{r-p-1}A_{\rm I}}
$$ 
for all $\lambda \in 
\Lambda(p,r)$ and $g \in \S_p$. 
As $\Lambdap(p,r)$ is a set of orbit representatives,
we deduce that
$$
X_{\rm I}^{(r)} = \sum_{\lambda\in \Lambdap(p,r)} 
\sum_{g \S_\lambda \in \S_p/\S_{\lambda}} X^{(g\cdot\lambda)}\equiv
\sum_{\lambda\in \Lambdap(p,r)}
|\S_p/\S_{\lambda}| X^{(\lambda)} \pmod{\F_{r-p-1} 
A_{\rm I}}.
$$
Now apply Lemma~\ref{L:zerolem} to complete the proof.
\end{proof}

\begin{Remark}\label{bladerunner}
When we meet the series of type I later on, the elements $X^{(r)}$
will satisfy the additional relations
\begin{equation}\label{s}
\big[X^{(r)}, X^{(s)}\big] = \sum_{t=r}^{s-1} X^{(t)} X^{(r+s-1-t)}
\end{equation}
for all $1 \leq r < s$.
Using this assumption, we will see indirectly that the elements
$X_{\rm I}^{(r)}$ can be expressed as polynomials in
the commuting elements $\big\{ (X^{(s)})^p\:\big|\:0 < s \leq \lfloor
  r / p \rfloor\big\}$;
see Remark~\ref{later1} below.
For example, when $p=2$, we have from the original definition that
$X_{\rm I}^{(3)} = X^{(1)} X^{(2)} + X^{(2)} X^{(1)}$,
which using (\ref{s})
may be simplified to obtain $X_{\rm I}^{(3)} = (X^{(1)})^2$.
Here are some other examples when $p=2$ under this assumption:
\begin{align*}
X_{\rm I}^{(0)} = X_{\rm I}^{(1)} &= 0,\\
X_{\rm I}^{(2)} = X_{\rm I}^{(3)} = X_{\rm I}^{(5)}  &= (X^{(1)})^2,\\
X_{\rm I}^{(4)} &= (X^{(2)})^2 + (X^{(1)})^2,\\
X_{\rm I}^{(6)} &= (X^{(3)})^2 + (X^{(2)})^2 + (X^{(1)})^2,\\
X_{\rm I}^{(7)} &= (X^{(3)})^2+(X^{(1)})^2,\\
X_{\rm I}^{(8)} &= (X^{(4)})^2 + (X^{(2)})^2 + (X^{(1)})^2.
\end{align*}
When $p > 2$ one can show
that $X_{\rm I}^{(p+1)} = 0$ assuming (\ref{s}).
\end{Remark}

\subsection{Series of type II}\label{S:typeII}
Fix $n \geq 1$ and a sequence $\left(d_r\right)_{r \geq 0}$ of natural numbers
with $d_0 = 0$. 
For $\mu \in \N^n$, define
\begin{equation}
\label{e:noptimal}
d_{\mu} := \sum_{i=1}^{n} d_{\mu_i}.
\end{equation}
We say that $r > 1$ is \emph{optimal} if
$d_{\mu}
\lneq
d_r$
for all 
$\mu \in \N^n$
such that either $|\mu| < r$, or $|\mu| = r$ and $\ell(\mu) > 1$.
The sole observation we need about this notion is the following:

\begin{Lemma}\label{L:optlemma}
Take $m \geq 1$ and suppose that the sequence $(d_r)_{r \geq
  0}$ is defined from
$d_r := 0$ for $r < m$ and 
$d_r  := m\lfloor \frac{r}{m} \rfloor - m$ for $r\geq m$. Then $mr$ is optimal for every $r > 1$. 
\end{Lemma}

\begin{proof}
Take $r > 1$.
We must show that
$d_{\mu} < mr-m$
for $\mu \in \N^n$
such that either $|\mu| < mr$,
or $|\mu|=mr$ and $\ell(\mu) > 1$.
Take any such composition $\mu$. 
Write each $\mu_i$ as $m a_i + b_i$ 
for $a_i \geq 0$ and $0 \leq b_i < m$.
If $a_i = 0$ for all $i$, then
$d_{\mu} = 0 < mr-m$.
If there is only one $i$ such that $a_i > 0$ then
$d_\mu = m a_i - m$ for this $i$.
This is less than $m r-m$ if $|\mu| < m r$ since
$m a_i \leq |\mu|$.
It is also less than $m r-m$ if $|\mu|=m r$ since in that case
$m a_i < |\mu|$ as $\mu$ has at least one other non-zero part
besides
$m a_i+b_i$.
Finally, if $a_i > 0$ for at least two different $i$, then
$$
d_\mu = 
\sum_{i\,\text{with}\,a_i > 0} \left(m a_i-m\right) 
<
\left(\sum_{i\,\text{with}\,a_i > 0} m a_i\right) - m
\leq m r -m,
$$
and we are done.
\end{proof}

For the remainder of the subsection, we assume in addition that
 $A_{\rm II}$ is
some filtered $\k$-algebra, 
and that we are given commuting elements
$\big\{X_i^{(r)} \in \F_{d_r} A_{\rm II}\:\big|\: 1 \leq i \leq n, r > 0\big\}$.
Setting $X_i^{(0)} := 1$ for all $i$,
which belongs to $\F_{d_0} A_{\rm II}$ as $d_0=0$,
consider the power series
\begin{eqnarray}
X_{\rm II}(u) = \sum_{r\geq 0} X_{\rm II}^{(r)} u^{-r} :=
\prod_{i=1}^{n} \left(\sum_{s \geq 0} X_{i}^{(s)} (u-i+1)^{-s}\right)
\in A_{\rm II}[[u^{-1}]].
\end{eqnarray}

\begin{Lemma}\label{Haha}
We have that
$$
X_{\rm II}^{(r)} = 
\sum_{\lambda \in \Lambda(n,r)}
\sum_{\mu \subseteq\lambda}
\left(
\prod_{i=1}^n \binom{\lambda_i-1}{\lambda_i-\mu_i}
(i-1)^{\lambda_i-\mu_i}
\right)
X^{(\mu)}
$$
where
$X^{(\mu)}:=X_1^{(\mu_1)} \cdots X_{n}^{(\mu_n)}$.
\end{Lemma}

\begin{proof}
By the binomial expansion,
\begin{align*}
(u - i+1)^{-s} 
&= u^{-s}(1 + (1-i) u^{-1})^{-s} =
 \sum_{t\geq 0} {\binom{-s}{t}} (1-i)^t u^{-s-t}\\
&= 
 \sum_{t\geq 0} {\binom{s+t-1}{t}} (i-1)^t u^{-s-t}
= 
 \sum_{t\geq s} {\binom{t-1}{t-s}} (i-1)^{t-s} u^{-t}.
 \end{align*}
It follows that
\begin{align*}
\sum_{s \geq 0} X_{i}^{(s)}(u-i+1)^{-s} &= 
\sum_{s \geq 0}\sum_{t \geq s}
\binom{t-1}{t-s} (i-1)^{t-s} X_i^{(s)}u^{-t}\\
&=\sum_{t \geq 0} \sum_{s=0}^t
\binom{t-1}{t-s} (i-1)^{t-s} X_i^{(s)}u^{-t}.
\end{align*}
Now take the product of this over $i=1,\dots,n$ then 
expand the $n$ parentheses, using the parts of $\lambda$ and $\mu$ to index the $t$-
and $s$-summations coming from each bracket, to deduce
$$
X_{\rm II}(u) = 
\sum_{\lambda \in \N^n}
\sum_{\mu \subseteq\lambda}
\left(
\prod_{i=1}^n \binom{\lambda_i-1}{\lambda_i-\mu_i}
(i-1)^{\lambda_i-\mu_i}
\right)
X^{(\mu)}u^{-|\lambda|}.
$$
It just remains to take the $u^{-r}$-coefficient.
\end{proof}

\begin{Lemma}
\label{L:type2}
We always have that $X_{\rm II}^{(1)} = X_1^{(1)}+\cdots+X_n^{(1)} \in
\F_{d_1} A_{\rm II}$.
Moreover, if $r > 1$ is optimal 
then
$X_{\rm II}^{(r)} \in \F_{d_r} A_{\rm II}$ and 
$X_{\rm II}^{(r)} \equiv X_1^{(r)} + \cdots + X_{n}^{(r)} 
\pmod{\F_{d_r-1} A_{\rm II}}.$
\end{Lemma}

\begin{proof}
The formula for $X_{\rm II}^{(1)}$ follows immediately from Lemma~\ref{Haha}.
Now assume that $r > 1$ is optimal.
We have that $X^{(\mu)} \in \F_{d_\mu} A_{\rm II}$.
So by the definition of optimal and Lemma~\ref{Haha}, we get that $X_{\rm II}^{(r)} \in
\F_{d_r} A_{\rm II}$. Moreover, the only terms in the summation from
that lemma which are not contained in $\F_{d_r-1}
A_{\rm II}$ arise when $\mu = \lambda$ and $\ell(\lambda) = 1$.
\end{proof}

\subsection{Series of type III}\label{S:typeIII}
The third type of power series we consider is defined similarly to type {\rm II}, however the resulting coefficients are slightly more complicated to describe
due to additional $\S_p$-symmetry.
Suppose that $A_{\rm III}$ is a filtered $\k$-algebra
containing commuting elements $\left\{X^{(r)} \in \F_{r-1} A_{\rm III} \:\big|\:
r > 0\right\}$.
Setting $X^{(0)} := 1 \in \F_0 A_{\rm III}$, consider
\begin{eqnarray}\label{sun}
X_{\rm III}(u) = \sum_{r\geq 0} X_{\rm III}^{(r)} u^{-r} := \prod_{i=1}^{p}\left(\sum_{r\geq 0} X^{(r)} (u - i+1)^{-r}\right).
\end{eqnarray}
Like in previous subsections, our goal is to obtain information
about the top degree component of the elements $X_{\rm III}^{(r)}$.
Henceforth, $X^{(\lambda)}$ will denote
$X^{(\lambda_1)} \cdots X^{(\lambda_p)}$ for 
$\lambda \in \Zg^p$; this is different to the usage in the previous subsection.
Notice that $X^{(\lambda)} \in \F_{|\lambda| - \ell(\lambda)} A_{\rm III}$.

\begin{Lemma}\label{smoke}
We have that
$$
X_{\rm III}^{(r)} = 
\sum_{s=0}^r \sum_{\mu \in \Lambdap(p,s)}
\sum_{\nu \in \Lambda(p,r-s)}
\sum_{g \S_\mu \in \S_p / \S_\mu}
\left(
\prod_{i=1}^p \binom{\mu_{g^{-1}(i)}+\nu_i-1}{\nu_i}
(i-1)^{\nu_i}
\right)
X^{(\mu)}.
$$
\end{Lemma}

\begin{proof}
Let $A_{\rm II}$ be the tensor product over $\k$ of $p$ copies of
$A_{\rm III}$, and define $$
X_i^{(r)} := 1^{\otimes(i-1)} \otimes X^{(r)} \otimes
 1^{\otimes(p-i)} \in A_{\rm II}.
$$
Applying Lemma~\ref{Haha} with $n=p$, 
we deduce that the power series
$$
X_{\rm II}(u) = \sum_{r \geq 0}X_{\rm II}^{(r)} u^{-r} :=
\prod_{i=1}^{p} \left(\sum_{r\geq 0} X_i^{(r)} (u -i+1)^{-r}\right) \in A_{\rm
  II}[[u^{-1}]]
$$
satisfies
$$
X_{\rm II}^{(r)} = 
\sum_{\lambda \in \Lambda(p,r)}
\sum_{\mu \subseteq\lambda}
\left(
\prod_{i=1}^p \binom{\lambda_i-1}{\lambda_i-\mu_i}
(i-1)^{\lambda_i-\mu_i}
\right)
X^{(\mu_1)} \otimes X^{(\mu_2)} \otimes\cdots\otimes X^{(\mu_p)}.
$$
Then we apply the linear map $A_{\rm II} \rightarrow A_{\rm
  III}$
defined by multiplying out the tensors. This maps $X_{\rm II}^{(r)} \mapsto X_{\rm
  III}^{(r)}$, so we deduce that
\begin{align*}
X_{\rm III}^{(r)} &= 
\sum_{\lambda \in \Lambda(p,r)}
\sum_{\mu \subseteq\lambda}
\left(
\prod_{i=1}^p \binom{\lambda_i-1}{\lambda_i-\mu_i}
(i-1)^{\lambda_i-\mu_i}
\right)
X^{(\mu)}\\
&=
\sum_{s=0}^r
\sum_{\mu \in \Lambda(p,s)}
\sum_{\nu \in \Lambda(p,r-s)}
\left(
\prod_{i=1}^p \binom{\mu_i+\nu_i-1}{\nu_i}
(i-1)^{\nu_i}
\right)
X^{(\mu)},
\end{align*}
where to get the second equation we switched the summations then replaced $\lambda$ by $\mu+\nu$.
Since $X^{(\mu)} = X^{(\mu')}$ if $\mu$ and $\mu'$ are in the
same $\S_p$-orbit, we can simplify this further to get
the final formula.
\end{proof}

We need to study the expression from Lemma~\ref{smoke} further.
Let $\k[x_1,\dots,x_p]^{\S_p}$ be the algebra of symmetric polynomials
over $\k$.
It is well known that this is freely generated by $\eps_1,\dots,\eps_p$, the
{\em elementary symmetric polynomials} defined from
$$
\eps_r = \eps_r(x_1,\dots,x_p) := \sum_{1 \leq i_1 < \cdots < i_r \leq p} x_{i_1} \cdots x_{i_r}.
$$
We also have the {\em power sums}
$$
\pi_r = \pi_r(x_1,\dots,x_p) := x_1^r + \cdots + x_p^r.
$$
These do not generate $\k[x_1,\dots,x_p]^{\S_p}$ since we are in
positive characteristic, but nevertheless
 every homogeneous symmetric polynomial of degree $< p$ can be
written as a polynomial in the power sums $\pi_1,\dots,\pi_{p-1}$. This follows by a
simple inductive argument from 
Newton's formula
which holds over the integers and hence over $\k$:
\begin{eqnarray}\label{e:newton}
k \eps_k = \sum_{i=1}^k (-1)^{i-1} \pi_i \eps_{k-i}.
\end{eqnarray}

\begin{Lemma}\label{L:vanishingpolys}
If
$f(x_1,\dots,x_p) \in \k[x_1,\dots,x_p]^{\S_p}$ is homogeneous 
of degree $0 < l < p-1$ then $f(0,1,\dots,p-1)=0$.
\end{Lemma}

\begin{proof}
There is nothing to do if $p=2$, so assume also
that $p > 2$.
By the remarks preceeding the lemma, it suffices to prove the claim that
$\pi_l(0,1,\dots,p-1)=0$
for $l=1,2,\dots,p-2$.
To see this, we appeal to the following famous identity, valid over $\Z$, which
is due to Pascal:
$$
\sum_{m=0}^l \binom{l+1}{m} (1^m + 2^m + \cdots + k^m) = (k+1)^{l+1} - 1$$
for all $k, l \in \Zg$. 
Substituting $k = p-1$ and working over $\k$ 
we get that
\begin{equation*}
\sum_{m=1}^l \binom{l+1}{m} \pi_m(0,1,\dots,p-1)
=
\sum_{m=1}^l \binom{l+1}{m} (1^m + 2^m + \cdots + (p-1)^m) = 0.
\end{equation*}
The claim follows easily from this by induction on $l=1,\dots,p-2$; 
one needs to note that $\binom{l+1}{m}
$ is non-zero in $\k$ for
$1\leq l \leq p-2$ and $m=1,...,l$.
\end{proof}

Now take $r \geq 0$ and $\mu \in \Lambdap(p,s)$ for some $0 \leq s
\leq r$, and define
\begin{equation}\label{not}
\gamma^{(r)}_\mu(x_1,\dots,x_p) :=
\sum_{\nu \in \Lambda(p,r-|\mu|)}
\sum_{g \S_\mu \in \S_p / \S_\mu}
\left(
\prod_{i=1}^p 
\binom{\mu_{g^{-1}(i)}+\nu_i-1}{\nu_i}
x_i^{\nu_i}
\right).
\end{equation}
This is relevant because by Lemma~\ref{smoke} we have that 
\begin{equation}\label{funny}
X_{\rm III}^{(r)} =
\sum_{s=0}^r 
\sum_{\mu \in \Lambdap(p,s)} 
\gamma^{(r)}_\mu(0,1,2,\dots,p-1) X^{(\mu)}.
\end{equation}
In fact, $\gamma^{(r)}_\mu$ belongs to $\k[x_1,\dots,x_p]^{\S_p}$:

\begin{Lemma}\label{tis}
$\gamma^{(r)}_\mu(x_1,\dots,x_p)$ is a homogeneous 
symmetric polynomial of degree $r-|\mu|$.
\end{Lemma}

\begin{proof}
The claim about degree is clear.
For $h \in \S_p$, we have that
\begin{align*}
h \cdot \gamma^{(r)}_\mu(x_1,\dots,x_p) &= 
\sum_{\nu \in \Lambda(p,r-|\mu|)}
\sum_{g \S_\mu \in \S_p / \S_\mu}
\left(
\prod_{i=1}^p 
\binom{\mu_{g^{-1}(i)}+\nu_i-1}{\nu_i}
x_{h(i)}^{\nu_i}
\right)\\
&=\sum_{\nu \in \Lambda(p,r-|\mu|)}
\sum_{g \S_\mu \in \S_p / \S_\mu}
\left(
\prod_{i=1}^p 
\binom{\mu_{(hg)^{-1}(i)}+\nu_i-1}{\nu_i}
x_{i}^{\nu_{i}}
\right)
\end{align*}
which equals 
$\gamma^{(r)}_\mu(x_1,\dots,x_p)$
because $h$ permutes the coset space $\S_p/\S_{\mu}$.
Hence, it is a symmetric polynomial.
\end{proof}

Now we can obtain our main result about the top degree component of
$X_{\rm III}^{(r)}$.

\begin{Lemma}\label{birds}
We have $X_{\rm III}^{(0)} = 1$, $X_{\rm III}^{(r)} = 0$ for $r =
1,\dots,p-1$, and
$$
X_{\rm III}^{(p)} = (X^{(1)})^p - X^{(1)} \in \F_0 A_{\rm III}.
$$
If $r > p$ and $p \mid r$
then 
$X_{\rm III}^{(r)} \in \F_{r-p} A_{\rm III}$ and
$$
X_{\rm III}^{(r)} \equiv (X^{(r/p)})^p - X^{(r-p + 1)}\pmod{\F_{r-p - 1} A_{\rm III}}.
$$
Finally, if $r > p$ and $p \nmid r$
then $X_{\rm III}^{(r)} \in \F_{r-p-1} A_{\rm III}$.
\end{Lemma}

\begin{proof}
Putting Lemmas~\ref{L:vanishingpolys}--\ref{tis} together shows that
\begin{equation*}
\gamma^{(r)}_\mu(0,1,\dots,p-1)=0
\qquad\text{whenever}\qquad
0 < r - |\mu| < p-1.
\end{equation*}
Also from (\ref{not}), it is clear that $\gamma^{(r)}_\mu = \delta_{r,0}$
in case $|\mu| = 0$. The case $|\mu| = r$ is also easy to understand:
we then have simply that $\gamma^{(r)}_\mu = |\S_p / \S_\mu|$, which is
$1$ if all parts of $\mu$ are equal and $0$ otherwise, thanks to Lemma~\ref{L:zerolem}.
So we can deduce from (\ref{funny}) that
\begin{equation}\label{funny2}
X_{\rm III}^{(r)} = 
\left\{
\begin{array}{ll}
1&\text{if $r=0$,}\\
\displaystyle
 (X^{(r/p)})^p + \sum_{s=1}^{r-p+1}\sum_{\mu \in \Lambdap(p,s)}
 \gamma^{(r)}_\mu(0,1,\dots,p-1) X^{(\mu)}&\text{if $p \mid r > 0$,}\\
\displaystyle\sum_{s=1}^{r-p+1}\sum_{\mu \in \Lambdap(p,s)}
\gamma^{(r)}_\mu(0,1,\dots,p-1) X^{(\mu)}&\text{if $p \nmid r > 0$.}
\end{array}\right.
\end{equation}
The lemma follows immediately from this in case $r < p$.
Now assume that $r \geq p$.
For $1 \leq s \leq r-p+1$ and $\mu \in \Lambda^+(p,s)$,
we have that $X^{(\mu)} \in \F_{|\mu|-\ell(\mu)} A_{\rm III}\subseteq 
\F_{r-p} A_{\rm III}$, showing that $X_{\rm
  III}^{(r)} \in \F_{r-p} A_{\rm III}$.
Moreover, $X^{(\mu)} \in \F_{r-p-1} A_{\rm III}$
unless $\mu = (r-p+1,0,\dots,0)$.
To complete the proof,  we show for this $\mu$ that
$$
\gamma_\mu^{(r)}(0,1,\dots,p-1)
= 
-\binom{r-1}{p-1},
$$
which is $0$ if $p \nmid r$ and $-1$ if $p \mid r$.

So $\mu = (r-p + 1, 0,\dots,0)$.
A set of representatives for $\S_p/\S_\mu$ is given by the $p$
distinct powers of the $p$-cycle $(1\:2\:\cdots \:p)$.
The product of binomial coefficients in the definition of
$\gamma_\mu^{(r)}(0,1,\dots,p-1)$
is non-zero only when $\nu$ has just one non-zero part, which is necessarily
equal to $p-1$. Moreover,
if this non-zero part is the $j$th part of $\nu$, we must have that $\mu_{g^{-1}(j)} =
r-p+1$ too, i.e. there is just one choice of $g$ that gives a
non-zero contribution.
We deduce that
$$
\gamma_\mu^{(r)}(0,1,\dots,p-1) = \sum_{j=1}^p
\binom{r-p+1+p-1-1}{p-1}(j-1)^{p-1}
\equiv-\binom{r-1}{p-1} \pmod{p}
$$
as claimed.
\end{proof}

\begin{Remark}\label{rain}
We will show later on that the elements $X_{\rm III}^{(r)}$ for $p \nmid r$ can be
expressed as polynomials in $\big\{X_{\rm III}^{(ps)}\:\big|\:
0 < s \leq \lfloor r/p\rfloor\big\}$; see Remark~\ref{later2}. It seems to be
hard to give a direct proof of this.
Here are some explicit examples which we computed by hand using (\ref{funny2}).
When $p=2$:
\begin{align*}
X_{\rm III}^{(1)} &= 0,\\
X_{\rm III}^{(2)} =X_{\rm III}^{(3)} = X_{\rm III}^{(5)} &= (X^{(1)})^2 + X^{(1)},\\
X_{\rm III}^{(4)} &= 
(X^{(2)})^2+X^{(3)}+ X^{(1)}X^{(2)}+X^{(2)} + X_{\rm III}^{(2)},\\
X_{\rm III}^{(6)} &= (X^{(3)})^2 + X^{(5)}
+ X^{(1)} X^{(3)}
+ X^{(1)} X^{(4)}
+X^{(2)}X^{(3)}
+X^{(3)}
+ X_{\rm III}^{(4)}.
\end{align*}
When $p=3$:
\begin{align*}
X_{\rm III}^{(1)} = X_{\rm III}^{(2)} = X_{\rm III}^{(4)} &= 0,\\
X_{\rm III}^{(3)} = X_{\rm III}^{(5)} &= (X^{(1)})^3 - X^{(1)},\\
X_{\rm III}^{(6)} &= (X^{(2)})^3 - X^{(4)} + X^{(1)} X^{(3)}-
(X^{(1)})^2 X^{(2)}+X^{(2)}-(X^{(2)})^2.
\end{align*}
Also, for all $p > 2$, we have that $X_{\rm III}^{(p+1)} = 0$.
\end{Remark}

\subsection{Series of type IV}
The fourth type is defined in almost the same way as type III.
So again we assume that $A_{\rm IV}$ is a filtered $\k$-algebra
containing commuting elements $\left\{X^{(r)} \in \F_{r-1} A_{\rm IV} \:\big|\:
r > 0\right\}$. However now we set $X^{(0)} := 0$, before defining
$$
X_{\rm IV}(u) = \sum_{r \geq 0} X_{\rm IV}^{(r)} u^{-r}
$$
by the same formula (\ref{sun}) which we used to define $X_{\rm III}(u)$.

The elements $X_{\rm IV}^{(r)}$ are given by the same formula
that was derived in
Lemma~\ref{smoke}
for the elements $X_{\rm III}^{(r)}$. However now for $\mu \in \Lambda(p,s)$
the monomial
$X^{(\mu)} = X^{(\mu_1)} \cdots X^{(\mu_p)}$ is zero
unless $\ell(\mu) = p$.
The following is an immediate consequence of Lemma~\ref{birds} using this observation.

\begin{Lemma}\label{lemIV}
We have that $X_{\rm IV}^{(r)} = 0$ for $r < p$
and
$X_{\rm IV}^{(p)} =(X^{(1)})^p \in \F_0 A_{\rm IV}.$
If $r > p$ and $p \mid r$
then 
$X_{\rm IV}^{(r)} \in \F_{r-p} A_{\rm IV}$ and
$X_{\rm IV}^{(r)} \equiv (X^{(r/p)})^p \pmod{\F_{r-p - 1} A_{\rm IV}}.$
Finally, if $r > p$ and $p \nmid r$
then $X_{\rm IV}^{(r)} \in \F_{r-p-1} A_{\rm IV}$.
\end{Lemma}

\begin{Remark}\label{rainier}
Like in Remark~\ref{rain}, we will see later that $X_{\rm IV}^{(r)}$ for $p \nmid r$ can be
expressed as a polynomial in 
$\big\{X_{\rm IV}^{(ps)}\:\big|\:
0 < s \leq \lfloor r/p\rfloor\big\}$; see Remark~\ref{rainer}.
\end{Remark}

\section{The shifted current algebra}\label{S:shiftedcurrentalg}

The shifted Yangian is a filtered deformation of the universal
enveloping algebra of a Lie algebra we call 
the {\em shifted current algebra}. In this section, we 
discuss this Lie algebra, describing the centre of its enveloping
algebra. Our notation follows \cite[ch. 2]{BK3}.

\subsection{The shift matrix}\label{shiftsec}
A \emph{shift matrix} is an $n \times n$ array $\sigma = (s_{i,j})_{1\leq i,j\leq n}$ of non-negative integers satisfying 
\begin{eqnarray}\label{e:shiftmatrix}
s_{i,j} + s_{j,k} = s_{i,k}
\end{eqnarray}
whenever $|i - j| + |j-k| = |i - k|$. It follows from the definition that $s_{i,i} = 0$ for $1 \leq i \leq n$,
and so $\sigma$ is entirely determined by the super-diagonal entries $s_{1,2}, s_{2,3},...,s_{n-1,n}$ and the sub-diagonal entries $s_{n,n-1},s_{n-1, n-2},...,s_{2,1}$.
We keep a choice of shift matrix fixed for the remainder of the
section.

\subsection{The shifted current algebra}
The \emph{current algebra} is the Lie algebra $\glnt := \gl_n
\otimes \k[t]$.
We will always denote this Lie algebra by $\g$
and write $U(\g)$ for its enveloping algebra and $S(\g)$ for the symmetric algebra.
When $x \in \gl_n$ and $f \in \k[t]$ we usually abbreviate $x\otimes f = xf \in \g$.
As a vector space, $\g$ is spanned by elements $\left\{ e_{i,j} t^r
  \:|\: 1 \leq i,j \leq n, r \geq 0\right\}$, and the Lie bracket is given by
\begin{eqnarray}\label{e:multloop}
\left[e_{i,j}t^r, e_{k,l} t^s\right] = \delta_{k,j} e_{i,l}t^{r+s} - \delta_{l,i} e_{k,j}t^{r+s}
\end{eqnarray}
where $1\leq i,j,k,l \leq n$ and $r,s\geq 0$.

For our fixed shift matrix $\sigma$, the \emph{shifted current algebra} is $\gs \subseteq \g$ spanned by
\begin{eqnarray}
\label{e:shiftcurrentspans}
\left\{e_{i,j} t^r \:\big|\: 1\leq i,j\leq n, r\geq s_{i,j}\right\}.
\end{eqnarray}

\begin{Lemma}
The shifted current algebra $\gs$ is a Lie subalgebra of $\g$, and it
is generated as a Lie algebra by
\begin{eqnarray}\label{e:shiftcurrentgens}
\left\{e_{i,i}t^r \:\big|\: 1 \leq i \leq n, r\geq 0\right\} \cup \left\{e_{i,i+1}t^r, e_{i+1,i}
t^s\:\big|\:1 \leq i < n, r \geq s_{i,i+1}, s \geq s_{i+1, i} \right\}.
\end{eqnarray}
\end{Lemma}
\begin{proof}
First we show that the span of \eqref{e:shiftcurrentspans} is closed under the bracket.
Let $1\leq i,j,k,l \leq n$, $r\geq s_{i,j}$ and $s \geq s_{k,l}$. 
By \eqref{e:multloop} we only need to check that $j=k$ implies $r + s \geq s_{i,l}$.
When $i < j < l$ or $l < j < i$ this follows from \eqref{e:shiftmatrix} so it remains to check the case $l < i < j$. Now we have
$r + s \geq s \geq s_{j,l} = s_{j,i} + s_{i,l} \geq s_{i,l}$. 

Now take $i < j$ and $r \geq s_{i,j}$. The fact that $e_{i,j}t^r$ lies in the algebra generated by \eqref{e:shiftcurrentgens}
can be proven by induction on $j - i$. A similar argument treats the
case that $i > j$, completing the proof.
\end{proof}

The adjoint action of $\gs$ on itself extends uniquely to actions of $\gs$
on $\Us$ and $\Ss$ by derivations.
The invariant subalgebras are denoted $\Us^{\gs}$ and $\Ss^{\gs}$, and the equality
$\Zs = \Us^{\gs}$ follows from general principles; e.g. see
\cite[2.4.9(i)]{Di}.

There is one obvious family of central elements in $\Us$.
For any $r \in \Zg$, we 
set
\begin{equation}
z_r := e_{1,1}t^r + \cdots + e_{n,n}t^r \in \gs.
\end{equation}
Then the set $\{z_r \:|\: r\geq 0\}$ forms a basis for the centre
$\z(\gs)$ of $\gs$, and
$\k[z_r\:|\:r \geq 0]$ is a subalgebra of $\Zs$.
The elements $z_r$ also define symmetric invariants in $\Ss^{\gs}$.

\subsection{Symmetric invariants}\label{sinv}
The current Lie algebra $\g$ has an obvious 
grading with $e_{i,j}t^r$ in degree $r$, and $\gs$ is a graded
subalgebra.
There is also a filtration 
\begin{equation}\label{filt1}
\Us = \bigcup_{r\geq 0}
\F_r \Us
\end{equation}
of the universal enveloping algebra $\Us$,
which is defined 
by placing $e_{i,j}t^r$ in degree $r+1$, i.e. $\F_r \Us$ is the span
of all monomials of the form $e_{i_1,j_1}t^{r_1}\cdots
e_{i_m,j_m}t^{r_m}$ 
with total degree $(r_1+1)+\cdots+(r_m+1) \leq r$.
The associated graded
algebra $\gr \Us$ is isomorphic 
(both as a graded algebra
and as a graded $\gs$-module) to $\Ss$ graded so that $e_{i,j}t^r$
is in degree $r+1$. 
We get induced an
inclusion
\begin{eqnarray}\label{grZinc}
\gr \Zs \subseteq \Ss^{\gs}.
\end{eqnarray}
In the remainder of the section, we are going to use this to compute
$\Zs$, revealing in particular that equality holds
in (\ref{grZinc}).
First, we must describe $\Ss^{\gs}$ explicitly.

\begin{Lemma}\label{symmetricinvariants}
The invariant algebra $\Ss^{\gs}$ is generated by 
$\{z_r \:|\: r\geq 0\}$
together with
$(\gs)^p := \{x^p \:|\: x\in \gs\} \subset \Ss$.
In fact, $\Ss^{\gs}$ is freely generated by
\begin{equation}\label{cgen}
\{z_r\:|\:r \geq 0\} \cup \big\{(e_{i,j}t^r)^p\:\big|\:
1 \leq i,j \leq n\text{ with }(i,j)\neq(1,1), r \geq s_{i,j}
\big\}.
\end{equation}
\end{Lemma}
\begin{proof}
Since we are in characteristic $p > 0$, we have that $(\gs)^p \subset \Ss^{\gs}$.
Let $I(\gs)$ be the subalgebra of $\Ss^{\gs}$ generated by $(\gs)^p$ and $\{z_r
\:|\: r \geq 0\}$. Let 
\begin{equation*}
B := \left\{(i,j,r) \:|\: 1 \leq i,j \leq n\text{ with }(i,j) \neq (1,1), r\geq s_{i,j}\right\}
\end{equation*}
for short.
Since the elements $\{z_r\:|\:r \geq 0\}\cup\big\{e_{i,j} t^r \:\big|\: r \geq 0, (i,j,r) \in B
\big\}$ give a basis for $\gs$, it follows that 
\begin{align*}
\Ss &= \k[z_r\:|\:r \geq 0]\left[e_{i,j}t^r \:\big|\: (i,j,r)\in
  B\right],\\
I(\gs) &= \k[z_r\:|\:r \geq 0]\left[ (e_{i,j}t^r)^p\:\big|\:(i,j,r) \in B\right],
\end{align*}
with both being free polynomial algebras.
Hence, $\Ss$ is free as an $I(\gs)$-module with basis 
$\left\{\prod_{(i,j,r) \in B} (e_{i,j} t^r)^{\omega(i,j,r)} \:\Big|\: \omega \in
  \Omega\right\}$, where
$$
\Omega := 
\left\{
\omega:B \rightarrow \N\:\Bigg|\:
\begin{array}{l}
0 \leq \omega(i,j,r) < p\text{ for all }(i,j,r)\in B,\\
\text{$\omega(i,j,r)=0$ for all but finitely many }(i,j,r)\in
B
\end{array}
\right\}.
$$
To complete the proof of the lemma, we must show that
$\Ss^{\gs} \subseteq I(\gs)$.
To do this, take $f \in \Ss^{\gs}$ and write it as
$$
f = \sum_{\omega\in\Omega} 
c_\omega \prod_{(i,j,r) \in B}(e_{i,j} t^r)^{\omega(i,j,r)}
$$
for $c_\omega \in I(\gs)$, all but finitely many of which are
zero. 
Also fix  a non-zero function $\omega \in \Omega$.
We must show that $c_\omega = 0$.

Suppose first that $\omega(i,j,r) > 0$ for some $(i,j,r) \in
B$ with $i \neq j$.
Choose $s \in \N$ that it is bigger than all
$r'$ such that $\omega(i',j',r') > 0$ for $(i',j',r')\in
B$.
Using the
Leibniz rule we see that 
\begin{multline*}
\ad(e_{i,i} t^s) (f) =
\sum_{\omega'\in\Omega} 
c_{\omega'} 
\sum_{\substack{(i',j',r')\in B \\ \omega(i',j',r')> 0}}
\omega(i',j',r')
(e_{i',j'} t^{r'})^{\omega(i',j',r')-1}
\left[e_{i,i} t^s, e_{i',j'} t^{r'}\right]\\\times
\prod_{\substack{(i'',j'',r'') \in B \\ (i'',j'',r'')\neq (i',j',r')}}
(e_{i'',j''} t^{r''})^{\omega(i'',j'',r'')}.
\end{multline*}
It is crucial for this that the coefficients $c_{\omega'}$ belong to
$I(\gs) \subseteq \Ss^{\gs}$.
Thanks to the choice of $s$,
the coefficient of
$$
(e_{i, j} t^{r})^{\omega(i,j,r)-1}e_{i, j} t^{r + s}
\prod_{\substack{(i'',j'',r'')\in B \\ (i'',j'',r'') \neq (i,j,r)}}
  (e_{i'',j''}t^{r''})^{\omega(i'',j'',k'')}
$$
in this expression is $c_\omega \omega(i,j,r)$. 
But also it must be zero since $f \in \Ss^{\gs}$.
Since $\omega(i,j,r)$ is non-zero in $\k$,
we deduce that $c_\omega = 0$ as required.

A very similar line of reasoning
treats the case that $\omega(j,j,r) > 0$ for some $(j,j,r)\in B$.
This time,
one picks $i \neq j$ (possible as $n > 1$)
and considers the coefficient of
$$(e_{j, j} t^{r})^{\omega(j,j,r)-1}e_{i, j} t^{r + s}
\prod_{\substack{(i'',j'',r'')\in B \\ (i'',j'',r'') \neq (j,j,r)}} 
(e_{i'',j''}t^{r''})^{\omega(i'',j'',r'')}$$
in
$\ad(e_{i, j} t^s)(f)$ for $s$ as before.
\end{proof}

\subsection{The restricted structure}
If $\r$ is any Lie algebra over $\k$ then a \emph{restricted structure on $\r$} is given by defining a \emph{$p$-map} $x \mapsto x^{[p]}$ on $\r$
such that the map
\begin{align*}
\xi : \r &\rightarrow U(\r),&
 x &\mapsto x^p - x^{[p]}
 \end{align*}
satisfies two properties: (i) $\xi(\r)$ lies in the centre of $U(\r)$; (ii) $\xi$ is $p$-semilinear. We remind the reader that $p$-semilinearity 
means that $\xi(\lambda x) = \lambda^p \xi(x)$ and $\xi(x + y) = \xi(x) + \xi(y)$, for all $x, y\in \r$, $\lambda \in \k$. It is readily
seen that such a map is determined by its effect on a basis.
\begin{Lemma}
The current algebra $\g$ is a restricted Lie algebra with $p$-map defined on the basis by the rule
$(e_{i,j}t^r)^{[p]} := \delta_{i,j} e_{i,j} t^{rp}.$
Moreover, $\gs$ is a restricted Lie subalgebra.
\end{Lemma}
\begin{proof}
Using Lemma~\ref{hop},
it is easy to see that when $\r$ is a restricted Lie algebra with $p$-map $x\mapsto x^{[p]}$
and $C$ is any commutative $\k$-algebra,
 the tensor product $\r \otimes C$ acquires the structure
of a restricted Lie algebra over $\k$ with multiplication and restricted structure given by
$[x \otimes a, y\otimes b] := [x,y] \otimes ab$ and $(x\otimes a)^{[p]} := x^{[p]} \otimes a^p$
for $x,y\in \r$ and $a,b\in C$. Since $\gl_n$ is restricted with $p$-map $e_{i,j}^{[p]} = \delta_{i,j}e_{i,j}$,
the first claim of the lemma follows. Clearly $x \mapsto x^p - x^{[p]}$ sends $\gs$ to $\Us$,
and so $\gs$ is a restricted subalgebra.
\end{proof}

\subsection{\boldmath The centre of $\Us$}\label{loopcentresec}
Using the restricted structure, we can define the
\emph{$p$-centre} $\Zps$ of $\Us$ to be the subalgebra of $\Zs$
generated by $x^p - x^{[p]}$ for all $x \in \gs$.
Since the $p$-map is $p$-semilinear, we have that
\begin{equation}\label{zdf}
\Zps = \k \left[\big(e_{i,j}t^r\big)^p - \delta_{i,j} e_{i,j}t^{rp} \:\Big|\: 1\leq
  i,j \leq n, r\geq s_{i,j}\right]
\end{equation}
as a free polynomial algebra.

\begin{Theorem}\label{loopcentre}
The centre $\Zs$ of $\Us$ is generated by $\{z_r\:|\:r \geq 0\}$ and
$\Zps$.
In fact, $\Zs$ is freely generated by
\begin{equation}\label{zgen}
\{z_r\:|\:r \geq 0\} \cup \big\{
(e_{i,j} t^r)^p - \delta_{i,j} e_{i,j} t^{rp}
\:\big|\:1 \leq i,j \leq n\text{ with }(i,j) \neq (1,1), r \geq s_{i,j}
\big\}.
\end{equation}
\end{Theorem}

\begin{proof}
It suffices to prove the second statement. 
Let $Z$ be the subalgebra of $\Zs$ generated by 
(\ref{zgen}).
For $r \geq 0$ we have that
$z_r \in \F_{r+1} \Us$ and
$$
\gr_{r+1} z_r = z_r \in \Ss.
$$
For $1\leq i,j \leq n$ and $r \geq s_{i,j}$ we have that
$(e_{i,j}t^r)^p - \delta_{i,j} e_{i,j} t^{rp}\in \F_{rp+p} \Us$
and
$$
\gr_{rp+p}
\left[(e_{i,j}t^r)^p - \delta_{i,j} e_{i,j} t^{rp}\right] = 
(e_{i,j}t^r)^p \in \Ss.
$$
So the elements (\ref{zgen}) are lifts of the algebraically independent generators of $\Ss^{\gs}$ from
(\ref{cgen}).
It follows that the elements (\ref{zgen}) are themselves algebraically
independent, and moreover 
$\Ss^{\gs} \subseteq \gr Z$.
Thanks to \eqref{grZinc}, we also have $\gr Z \subseteq \gr \Zs
\subseteq \Ss^{\gs}$,
so equality must hold throughout:
$\gr Z = \gr \Zs
= \Ss^{\gs}.$
This implies
$Z = \Zs$.
\end{proof}

\begin{Remark}
If $\k$ is algebraically closed of characteristic $p > 0$ and $\g = \Lie(G)$ for a reductive
algebraic $\k$-group $G$ 
satisfying standard hypotheses,
then it is well known that the centre of $U(\g)$ is generated by the
$p$-centre and $U(\g)^G$. 
Theorem~\ref{loopcentre} can be seen as an analogue of this
classical result in the context of loop algebras.
\end{Remark}

\section{Modular Yangians and shifted Yangians}

In this section, we define and study the Yangian $Y_n$ and its shifted subalgebra
$Y_n(\sigma)$ in positive characteristic. In particular, we prove that the RTT
presentation for $Y_n$ from the introduction is equivalent to a
slightly modified version of the Drinfeld presentation. 

\subsection{The RTT generators}
We define the Yangian $Y_n$ over $\k$ as in the introduction. So it has generators
$\big\{T_{i,j}^{(r)}\:\big|\: 1 \leq i,j \leq n, r > 0\big\}$
subject just to the relations
(\ref{rttnew}). Recall also that $T_{i,j}^{(0)} := \delta_{i,j}$.
By the same proof as \cite[Proposition~1.2]{MNO}, the following give
an equivalent set of defining relations:
\begin{eqnarray}\label{e:RTTrelations}
\left[T_{i,j}^{(r)}, T_{k,l}^{(s)}\right] = \sum_{t=0}^{\min(r,s) - 1}
\left( T_{k, j}^{(t)} T_{i,l}^{(r+s-1-t)} - 
T_{k,j}^{(r+s-1-t)} T_{i,l}^{(t)}\right)
\end{eqnarray}
for $1\leq i,j,k,l \leq n$ and $r,s > 0$.
Using this and induction on $r+s$, it is easy to see in particular
that
\begin{equation}\label{Theycommute}
T_{i,j}^{(r)} T_{i,j}^{(s)} = T_{i,j}^{(s)} T_{i,j}^{(r)}
\end{equation}
for all $r,s \geq 0$.

We often put the generators $T_{i,j}^{(r)}$ for all $r \geq 0$
together to form the power
series
$T_{i,j}(u)$ as in (\ref{amazing}). 
Then these power series for all $1 \leq i,j \leq n$ can be collected together into a single matrix
$T(u) := (T_{i,j}(u))_{1 \leq i,j \leq n} \in M_n(Y_n[[u^{-1}]])$.
Using these matrices,
the relations can be expressed
in an extremely compact form, known as the RTT presentation. Since this presentation will play no role in our work,
we do not bother going through the details, rather we refer the reader
to \cite[Proposition~1.8]{MNO} where the calculations apply equally well over $\k$ as they do over $\C$.

It will not play a significant role in this paper, but we should
mention that $Y_n$ is a Hopf algebra. Its comultiplication $\Delta$
and antipode $S$ are
given by
\begin{align}\label{comult}
\Delta(T_{i,j}(u)) &= \sum_{k=1}^n 
T_{i,k}(u) \otimes T_{k,j}(u),
&
S(T_{i,j}(u)) &= \widetilde{T}_{i,j}(u),
\end{align}
where $\widetilde{T}_{i,j}(u)$ is the $(i,j)$-entry of the matrix $T(u)^{-1}$.
The counit sends $T_{i,j}(u) \mapsto \delta_{i,j}$.

\subsection{Filtrations and the PBW theorem}
There are two distinguished filtrations on $Y_n$. 
The first one, called the {\em canonical filtration} in
\cite[\textsection 2]{BK1}, 
places $T_{i,j}^{(r)}$ in degree $r$. The associated graded algebra is
commutative.
Since this filtration will not play a significant role
in this article, we will not reserve any special notation for it. In
fact, will never mention it again after the proof of the following
fundamental {\em
  PBW theorem}.

\begin{Theorem}\label{pbw}
Ordered monomials
in the generators $\big\{T_{i,j}^{(r)}\:\big|\:1 \leq i,j \leq n, r >
0\big\}$ taken in any fixed order give a basis for $Y_n$.
\end{Theorem}

\begin{proof}
It is easy to see that these monomials span, e.g. one can argue by induction on degree in
the canonical filtration. To show that they are linearly independent,
we may assume that $\k$ is algebraically closed.
Then the proof given in characteristic zero in \cite[Theorem 3.1]{BK1} works just as well in
positive characteristic.
\end{proof}

The second filtration on $Y_n$, called \emph{the
loop filtration} in \cite[\textsection 2]{BK1}, will be ubiquitous. We denote it by
\begin{equation}\label{filt2}
Y_n = \bigcup_{r \geq 0} \F_r Y_n.
\end{equation}
It is defined by placing $T_{i,j}^{(r)}$ in degree $r-1$, i.e.
$\F_r Y_n$ is the span of all monomials of the form
$T_{i_1,j_1}^{(r_1)} \cdots T_{i_m,j_m}^{(r_m)}$ with $
(r_1-1)+\cdots+(r_m-1) \leq r$. 
We warn the reader that the notation $\F_r Y_n$ is often used
elsewhere in the literature for
the canonical filtration. Also (\ref{filt2}) should
not be confused with the filtration on $U(\g)$ from (\ref{filt1}); the
latter will never be used again.

To describe the associated graded algebra $\gr Y_n$, 
recall that 
$\g = \gl_n[t]$ is generated by $\{e_{i,j}t^r\:|\:1 \leq i,j \leq n, r
\geq 0\}$ subject to relations \eqref{e:multloop},
and notice by the definition that $\gr Y_n$ is generated by elements 
$\big\{\gr_{r-1} T_{i,j}^{(r)}\:\big|\:
1\leq i,j \leq n, r> 0\big\}$.

\begin{Lemma}\label{L:loopPBW}
There is an isomorphism
$\chi:U(\g) \isoto \gr Y_n$ sending $e_{i,j}t^{r} \mapsto \gr_{r} T_{i,j}^{(r+1)}$ 
for each $1 \leq i,j \leq n$ and $r \geq 0$.
\end{Lemma}

\begin{proof}
Relation (\ref{e:RTTrelations}) implies that
\begin{align*}
[\gr_{r} T_{i,j}^{(r+1)}, \gr_{s} T_{k,l}^{(s+1)}] &= [T_{i,j}^{(r+1)}, T_{k,l}^{(s+1)}] +
\F_{r+s-1} Y_n\\ 
&= \delta_{k,j} T_{i,l}^{(r+s+1)} - \delta_{i,l} T_{k,j}^{(r+s+1)} + \F_{r+s-1} Y_n \\
&= \delta_{k,j} \gr_{r+s} T_{i,l}^{(r+s+1)} - \delta_{i,l}
\gr_{r+s} T_{k,j}^{(r+s+1)}.
\end{align*}
Comparing with 
(\ref{e:multloop}), we deduce that the map in the statement of the
lemma is well defined.
To see that it is an isomorphism, 
one uses the PBW basis from Theorem~\ref{pbw} to see that a basis for $U(\g)$
is sent to a basis for $\gr Y_n$.
\end{proof}

\subsection{Drinfeld-type presentation}
Since the leading minors of the $n \times n$ matrix $T(u)$ are invertible, 
it possesses a Gauss factorization
\begin{equation}\label{gfact}
T(u) = F(u) D(u) E(u)
\end{equation}
for unique matrices
$$
D(u) = \left(
\begin{array}{cccc}
D_{1}(u) & 0&\cdots&0\\
0 & D_{2}(u) &\cdots&0\\
\vdots&\vdots&\ddots&\vdots\\
0&0 &\cdots&D_{n}(u)
\end{array}
\right),
$$$$
E(u) = 
\left(
\begin{array}{cccc}
1 & E_{1,2}(u) &\cdots&E_{1,n}(u)\\
0 & 1 &\cdots&E_{2,n}(u)\\
\vdots&\vdots&\ddots&\vdots\\
0&0 &\cdots&1
\end{array}
\right),\:
F(u) = \left(
\begin{array}{cccc}
1 & 0 &\cdots&0\\
F_{1,2}(u) & 1 &\cdots&0\\
\vdots&\vdots&\ddots&\vdots\\
F_{1,n}(u)&F_{2,n}(u) &\cdots&1
\end{array}
\right).
$$
This defines power series
\begin{align*}
D_i(u) &= \sum_{r \geq 0} D_i^{(r)} u^{-r},&
 E_{i,j}(u) &= \sum_{r \geq 0} E_{i,j}^{(r)} u^{-r},&
F_{i,j}(u) &= \sum_{r\geq 0} F_{i,j}^{(r)} u^{-r}
\end{align*}
in $Y_n[[u^{-1}]]$ with $D_i^{(0)} = 1$ and $E_{i,j}^{(0)} =
F_{i,j}^{(0)} = 0$.
Let 
\begin{align*}
E_i(u) = \sum_{r \geq 0} E_i^{(r)} u^{-r}
&:= E_{i,i+1}(u),&
F_i(u) = \sum_{r \geq 0} F_i^{(r)} u^{-r}&:= 
F_{i,i+1}(u)
\end{align*} 
for short.
Also let 
\begin{equation*}
\widetilde D_i(u) =\sum_{r \geq 0} \widetilde{D}_i^{(r)} u^{-r}
:= D_i(u)^{-1}.
\end{equation*} 
We warn the reader that our $\widetilde{D}_i(u)$  differs by a sign from one
used in \cite{BK1, BK2, BK3}.
This accounts for several other sign differences
in the exposition below compared to {\em loc. cit.}, e.g.
we have that $\widetilde D_0^{(0)} = 1$. 

In terms of quasi-determinants of \cite{GR2}, 
we have the following more explicit
descriptions, as noted already in \cite[\textsection 5]{BK1}:
\begin{equation}\label{qd1}
\!\!\!\!\!\!\!\!\!\!\!\!\!\!\!\!\!\!\!\!
D_i(u) = 
\left|
\begin{array}{cccc}
T_{1,1}(u) & \cdots & T_{1,i-1}(u)&T_{1,i}(u)\\
\vdots & \ddots &\vdots&\vdots\\
T_{i-1,1}(u)&\cdots&T_{i-1,i-1}(u)&T_{i-1,i}(u)\\
T_{i,1}(u) & \cdots & T_{i,i-1}(u)&
\hbox{\begin{tabular}{|c|}\hline$T_{i,i}(u)$\\\hline\end{tabular}}
\end{array}
\right|,
\end{equation}
\begin{equation}
\label{qd2}
E_{i,j}(u) = 
\widetilde{D}_i(u) \left|
\begin{array}{cccc}
T_{1,1}(u) & \cdots &T_{1,i-1}(u)& T_{1,j}(u)\\
\vdots & \ddots &\vdots&\vdots\\
T_{i-1,1}(u) & \cdots & T_{i-1,i-1}(u)&T_{i-1,j}(u)\\
T_{i,1}(u) & \cdots & T_{i,i-1}(u)&
\hbox{\begin{tabular}{|c|}\hline$T_{i,j}(u)$\\\hline\end{tabular}}
\end{array}
\right|,
\end{equation}
\begin{equation}\label{qd3}
F_{i,j}(u) = 
\left|
\begin{array}{cccc}
T_{1,1}(u) & \cdots &T_{1,i-1}(u)& T_{1,i}(u)\\
\vdots & \ddots &\vdots&\vdots\\
T_{i-1,1}(u) & \cdots & T_{i-1,i-1}(u)&T_{i-1,i}(u)\\
T_{j,1}(u) & \cdots & T_{j,i-1}(u)&
\hbox{\begin{tabular}{|c|}\hline$T_{j,i}(u)$\\\hline\end{tabular}}
\end{array}
\right|\widetilde{D}_i(u).
\end{equation}
Since $E_{j-1}^{(1)} = T_{j-1,j}^{(1)}$ and $F_{j-1}^{(1)} = T_{j,j-1}^{(1)}$,
it follows easily that
\begin{equation}\label{eij}
E_{i,j}^{(r)} = 
\big[E_{i,j-1}^{(r)}, E_{j-1}^{(1)}\big],\qquad
F_{i,j}^{(r)} = 
\big[F_{j-1}^{(1)},F_{i,j-1}^{(r)}\big].
\end{equation}
for $i+1 < j \leq n$. 

Now we state the main theorem of the section,
which is
the modular analogue of \cite[Theorem~5.2]{BK1}. Although not written
explicitly there in this form, this presentation must surely have been known to
Drinfeld when writing \cite{D}. 

\begin{Theorem}\label{drinpres} The algebra $Y_n$ is generated by the elements
$\big\{D_i^{(r)}, \widetilde{D}_i^{(r)}\:\big|\:1 \leq i \leq n, r > 0\big\}$
and 
$\big\{E_i^{(r)}, F_i^{(r)}\:\big|\:1 \leq i < n,r > 0\big\}$
subject only to the following relations:
\begin{align}
\big[D_i^{(r)}, D_j^{(s)}\big] &=  0,\label{r2}\\
\big[E_i^{(r)},F_j^{(s)}\big] &= -\delta_{i,j} 
\sum_{t=0}^{r+s-1} D_{i+1}^{(r+s-1-t)}\widetilde D_{i}^{(t)},\label{r3}\\
\big[D_i^{(r)}, E_j^{(s)}\big] &= (\delta_{i,j}-\delta_{i,j+1})
\sum_{t=0}^{r-1} D_i^{(t)} E_j^{(r+s-1-t)},\label{r4}\\
\big[D_i^{(r)}, F_j^{(s)}\big] &= (\delta_{i,j+1}-\delta_{i,j})
\sum_{t=0}^{r-1} F_j^{(r+s-1-t)}D_i^{(t)} ,\label{r5}\\
\big[E_i^{(r)}, E_i^{(s)}\big] &=
\sum_{t=r}^{s-1} E_i^{(t)} E_i^{(r+s-1-t)}
\hspace{11mm}\text{if $r < s$},\label{r6}\\
\big[F_i^{(r)}, F_i^{(s)}\big] &=
\sum_{t=s}^{r-1} 
F_i^{(r+s-1-t)} F_i^{(t)}\hspace{11mm}\text{if $r > s$},\label{r7}\\
\big[E_i^{(r+1)}, E_{i+1}^{(s)}\big]&- 
 \big[E_i^{(r)}, E_{i+1}^{(s+1)}\big]=
E_i^{(r)} E_{i+1}^{(s)},\label{r8}\\
\big[F_i^{(r)}, F_{i+1}^{(s+1)}\big]&- \big[F_i^{(r+1)}, F_{i+1}^{(s)}\big] =
 F_{i+1}^{(s)} F_i^{(r)},\label{r9}\\
\big[E_i^{(r)}, E_j^{(s)}\big] &= 0 \hspace{37.9mm}\text{ if }|i-j|> 1,\label{r10}\\
\big[F_i^{(r)}, F_j^{(s)}\big] &= 0 \hspace{37.9mm}\text{ if }|i-j|>
1,\label{r11}
 \end{align}\begin{align}
\Big[E_i^{(r)}, \big[E_i^{(s)}, E_j^{(t)}\big]\Big] &+ 
\Big[E_i^{(s)}, \big[E_i^{(r)}, E_j^{(t)}\big]\Big] = 0 \quad\text{ if }|i-j|=1, r
\neq s,\label{r12}\\
\Big[F_i^{(r)}, \big[F_i^{(s)}, F_j^{(t)}\big]\Big] &+ 
\Big[F_i^{(s)}, \big[F_i^{(r)}, F_j^{(t)}\big]\Big] = 0 \quad\text{ if
}|i-j|=1,r \neq s\label{r13}\\
\Big[E_i^{(r)}, \big[E_i^{(r)}, E_j^{(t)}\big]\Big] &= 0 \hspace{39.2mm}\text{ if } |i - j| = 1, \label{rel15}\\
\Big[F_i^{(r)}, \big[F_i^{(r)}, F_j^{(t)}\big]\Big] &= 0 \hspace{39.2mm}\text{ if } |i - j| = 1,\label{rel16} 
\end{align}
for all admissible $i,j,r,s, t$.
In these relations, we use the shorthands
$D_i^{(0)} = \widetilde{D}_i^{(0)} := 1$, 
and the elements $\widetilde{D}_i^{(r)}$ for $r > 0$ 
are defined recursively
by
$\widetilde{D}_i^{(r)} := -\sum_{t=1}^r D_i^{(t)} \widetilde D_i^{(r-t)}$.
\end{Theorem}

\begin{Remark}\label{differencerem}
The relations \eqref{r2}--\eqref{r11} are the same as relations
(5.9)--(5.18) of \cite{BK1}, 
however relations (5.19)--(5.20) of {\it loc. cit.}
are expressed here as the four relations \eqref{r12}--\eqref{rel16}. 
This is essential in charactersitic 2.
\end{Remark}

\begin{proof}
This is very similar to the proof in
the characteristic zero explained in \cite{BK1}. We just give a
brief account in order to highlight the minor differences.

To start with, we need to check that the relations
\eqref{r2}--\eqref{rel16} do indeed hold in $Y_n$.
 For all but the
last two relations, this is carried out already in \cite{BK1} over the
ground field $\mathbb{C}$. 
The arguments there start from the power series form
of the defining relation (\ref{rttnew}), namely, that
\begin{equation}
(u-v) [T_{i,j}(u), T_{k,l}(v)] = T_{k,j}(u) T_{i,l}(v) - T_{k,j}(v)
T_{i,l}(u)
\end{equation}
in $Y_n[[u^{-1}, v^{-1}]]$, then extract from this various
relations satisfied by the power series $D_i(u), E_i(u)$ and
$F_i(u)$; some of these are recorded in Lemma~\ref{powerrelations1} 
below. Then the desired relations follow by computing
various coefficients in these power series relations. These
calculations
can be performed without any difference in positive characteristic, yielding all
of our relations except for (\ref{rel15})--(\ref{rel16}).
To establish (\ref{rel15}), we use the power series relation
\begin{equation}\label{star}
[E_i(u), [E_i(u), E_j(v)]] = 0
\end{equation}
for $|i-j|=1$
which is proved by reducing first to the case $\{i,j\} = \{1,2\}$ then
arguing as in \cite[Lemma 5.6]{BK1}. 
Taking the $u^{-2r} v^{-t}$-coefficient in this identity and using also
(\ref{r12}) gives (\ref{rel15}). The proof of (\ref{rel16}) is similar.

Now let $\widehat{Y}_n$ be the algebra with generators 
$\big\{D_i^{(r)}, \widetilde{D}_i^{(r)}\:\big|\:1 \leq i \leq n, r \geq 0\big\}$
and 
$\big\{E_i^{(r)}, F_i^{(r)}\:\big|\:1 \leq i < n,r > 0\big\}$
subject to all of the relations recorded in the statement of the
theorem.
The previous paragraph implies that there is a well-defined
homomorphism $\theta:\widehat{Y}_n \rightarrow Y_n$ taking the generators of $\widehat{Y}_n$ to
the elements of $Y_n$ with the same names.
Define higher root elements $E_{i,j}^{(r)},
F_{i,j}^{(r)} \in \widehat{Y}_n$ for all $1 \leq i < j \leq n$ and $r
> 0$ by setting $E_{i,i+1}^{(r)} := E_i^{(r)}$ and
$F_{i,i+1} := F_i^{(r)}$, then using the formula (\ref{eij})
inductively when $|j-i|>1$.
This definition ensures that $\theta$ sends $E_{i,j}^{(r)}, F_{i,j}^{(r)}\in\widehat{Y}_n$
to the elements of $Y_n$ with the same names.

To complete the proof, we define a filtration $\widehat{Y}_n=\bigcup_{r \geq 0}
\F_r \widehat{Y}_n$ by declaring
that the elements $D_i^{(r)}, E_{i,j}^{(r)}$ and $F_{i,j}^{(r)}$ 
defined in the previous paragraph are of filtered
degree $r-1$. 
We claim that
there is a surjective graded algebra homomorphism
\begin{equation}\label{hydra}
\psi:U(\g) \twoheadrightarrow \gr \widehat{Y}_n,
\quad
e_{i,j} t^{r} \mapsto e_{i,j;r}
:=\left\{
\begin{array}{ll}
\gr_{r} D_i^{(r+1)}&\text{if $i=j$,}\\
\gr_{r} E_{i,j}^{(r+1)}&\text{if $i<j$,}\\
\gr_{r} F_{j,i}^{(r+1)}&\text{if $i>j$,}
\end{array}\right.
\end{equation}
for $1 \leq i,j \leq n$ and $r \geq 0$.
To see this, surjectivity is immediate from the way the filtration is defined, so
it suffices to show that the image of the defining relation (\ref{e:multloop}) for $U$ holds
in $\gr \widehat{Y}_n$, i.e.
$[e_{i,j;r}, e_{k,l;s}] = \delta_{k,l} e_{i,l;r+s}-\delta_{i,l}
e_{k,j;r+s}$
for $1 \leq i,j,k,l \leq n$ and $r,s \geq 0$.
There are six cases: (a) $i=j,k=l$; 
(b) $i = j, k < l$; (c) $i=j, k > l$; 
(d) $i < j, k < l$; (e) $i > j, k > l$; (f) $i < j, k > l$.
For these, (a) is immediate from (\ref{r3}), (c) is a similar argument
to (b), and (e) is a similar argument to (d). So we just prove (b),
(d) and (f).
The proof of (d) is the same argument as in the proof of \cite[Lemma~5.8]{BK1}
except when $r=s$ and $|i-j|=1$,
when one must replace the relation
$[e_{i,i+1;r}, [e_{i,i+1; s}, e_{j,j+1; t}]] = - [e_{i,i+1;s},
[e_{i,i+1; r}, e_{j,j+1; t}]]$ with $[e_{i,i+1;r}, [e_{i,i+1; r},
e_{j,j+1; s}]] = 0$, which follows straight from \eqref{rel15}.
Also (b) follows when $|k-l|=1$ using \eqref{r4}; then it may be
deduced in general using (d) and induction on $k-l$.
The argument for (f) is similar: it follows when $|i-j|=|k-l|=1$ using \eqref{r3}; then
it follows in general using (d), (e) and induction.

Using (\ref{qd1})--(\ref{qd3}), one sees that $D_i^{(r)}, E_{i,j}^{(r)},
F_{i,j}^{(r)} \in \F_{r-1} Y_n$, so that $\theta$ is a filtered
homomorphism. Moreover, the following diagram commutes:
$$
\begin{diagram}
\node[2]{U(\g)}\arrow{sw,t,A}{\psi}\arrow{se,t}{\chi}\\
\node{\gr \widehat{Y}_n}\arrow[2]{e,t}{\gr
  \theta}\node[2]{\gr Y_n.}
\end{diagram}
$$
We already showed that $\chi$ is an isomorphism
in Lemma~\ref{L:loopPBW}. It follows that $\gr \theta$ is one too.
Hence, $\theta$ is an isomorphism as required.
\end{proof}

Henceforth, we will {\em identify} $U(\g)$ and $\gr Y_n$ via the isomorphism
$\chi$ from Lemma~\ref{L:loopPBW}, i.e. we identify 
$e_{i,j}t^{r}$
with $\gr_{r} T_{i,j}^{(r+1)}$.
The proof of Theorem~\ref{drinpres} shows moreover
that  $D_i^{(r+1)}, E_{i,j}^{(r+1)}$ and $F_{i,j}^{(r+1)}$ all belong to
$\F_{r} Y_n$, and under our identification we have that
\begin{equation}\label{stuff}
e_{i,j} t^{r} 
=\left\{
\begin{array}{ll}
\gr_{r} D_i^{(r+1)}&\text{if $i=j$,}\\
\gr_{r} E_{i,j}^{(r+1)}&\text{if $i<j$,}\\
\gr_{r} F_{j,i}^{(r+1)}&\text{if $i>j$.}
\end{array}\right.
\end{equation}
Using this and the PBW theorem for $U(\g)$, we obtain
the following giving a PBW basis for $Y_n$ in terms of
the Drinfeld generators.

\begin{Theorem}\label{L:loopPBW2}
Ordered monomials
in the elements 
\begin{equation}
\big\{D_i^{(r)} \:\big|\: 1\leq i \leq n, r > 0\big\} \cup
\big\{E_{i,j}^{(r)}, F_{i,j}^{(r)}\:\big|\: 1\leq i < j \leq n, r >
0\big\}
\end{equation}
taken in any fixed ordering
form a basis for $Y_n$.
\end{Theorem}

We will need the power series forms of some of
the relations from Theorem~\ref{drinpres}.
We record these and some consequences next.

\begin{Lemma}\label{powerrelations1}
The following equalities hold in $Y_n[[u^{-1}, v^{-1}]]$:
\begin{align}
\label{prel3}
(u-v) [E_i(u), F_i(v)] &= D_{i+1}(u)\tD_i(u)  - D_{i+1}(v)\tD_i(v),\\ \label{prel4}
(u-v) [E_i(u), E_i(v)] &= (E_i(v) - E_i(u))^2, \\ \label{prel1}
(u-v) [E_i(u), D_i(v)] &= D_i(v) (E_i(u) - E_i(v)),\\
 (u-v)[E_i(u), \tD_{i+1}(v)] &= (E_i(u) -
E_i(v)) \tD_{i+1}(v),\label{jons}\\
\label{prel5}
(u-v) [E_i(u), \tD_i(v)] &=  (E_i(v) - E_i(u)) \tD_i(v),\\
\label{prel2}
(u-v) [E_i(u), D_{i+1}(v)] &= D_{i+1}(v) (E_i(v) - E_i(u)).
\end{align}
\end{Lemma}

\begin{proof}
Equations \eqref{prel3}--\eqref{jons} were proven over $\mathbb{C}$ in \cite[Lemma~5.4]{BK1}; 
the same proof works here.
Then \eqref{prel5}--\eqref{prel2} follow from
\eqref{prel1}--\eqref{jons} using $D_j(v) \tD_j(v) =1$.
\end{proof}

\begin{Corollary}\label{violin}
The following hold in $Y_n[[u^{-1}]]$:
\begin{align}
E_i(u-1) D_i(u) &= D_i(u) E_i(u),&
\widetilde{D}_i(u)
E_i(u-1) &= E_i(u) \widetilde{D}_i(u),\label{jon2}\\
D_{i+1}(u) E_i(u) &= E_i(i+1) D_{i+1}(u), &
E_i(u)\widetilde{D}_{i+1}(u)  &= \widetilde{D}_{i+1}(u) E_i(i+1).\label{jons3}
\end{align}
\end{Corollary}

\begin{proof}
These follow from the identities in the previous lemma by specializing
$v$.
For example, to get the first relation in (\ref{jon2}), set $v:=u+1$
in (\ref{prel1}), simplify, then replace $u$ by $u-1$.
\end{proof}

\begin{Lemma}\label{le}
The following relations hold in $Y_n[[u^{-1}, v^{-1}]]$
for all $m \geq 0$:
\begin{align}
(u-v)[E_i(u), (E_i(v)-E_i(u))^m] &=
m (E_i(v)-E_i(u))^{m+1},\label{new1}\\
(u-v)[E_i(u), D_i(v)(E_i(v)-E_i(u))^m] &= (m-1) D_i(v)(E_i(v)-E_i(u))^{m+1},\label{new2}\\
 (u-v)[E_i(u), D_{i+1}(v)(E_i(v)-E_i(u))^m] &= (m+1) D_{i+1}(v)(E_i(v)-E_i(u))^{m+1},\label{new3}\\
 (u-v)[E_i(u), D_{i+1}(v)(E_i(v)\!-\!E_i(u))^m\tilde{D}_i(v)] &= (m+2)
 D_{i+1}(v)(E_i(v)\!-\!E_i(u))^{m+1} \tilde D_i(v).\label{new4}
\end{align}
\end{Lemma}

\begin{proof}
The relation (\ref{new1}) follows from (\ref{prel4}) and the Leibniz
rule.
Then (\ref{new2})--(\ref{new4}) follow from (\ref{new1}),
(\ref{prel1}), (\ref{prel5}) and (\ref{prel2}) using Leibniz again.
\end{proof}

The following relations are closely related to the ones in
Lemma~\ref{le}, but it is easier to prove them from scratch.

\begin{Lemma}\label{lf}
For any $i = 1,\dots,n-1$, $m \geq 0$ and $r,s > 0$, we have that
\begin{align}\label{epic1}
\Bigg[E_i^{(r)}, \hspace{-6mm}\sum_{\substack{s_1,\dots,s_m \geq r \\ s_1+\cdots+s_m =
    (m-1)(r-1)+s}}
\hspace{-10mm}
E_i^{(s_1)}
\cdots
E_i^{(s_m)}\Bigg]
&= m
\hspace{-6mm}\sum_{\substack{s_1,\dots,s_{m+1} \geq r \\ s_1+\cdots+s_{m+1} =
    m(r-1)+s}}
\hspace{-10mm}E_i^{(s_1)}
\cdots
E_i^{(s_{m+1})},\\
\Bigg[E_i^{(r)}, \hspace{-6mm}\sum_{\substack{s_1,\dots,s_m \leq r-1 \\ s_1+\cdots+s_m =
    (m-1)(r-1)+s}}\hspace{-10mm}
E_i^{(s_1)}
\cdots
E_i^{(s_m)}\Bigg]
&= -m\hspace{-6mm}
\sum_{\substack{s_1,\dots,s_{m+1} \leq r-1 \\ s_1+\cdots+s_{m+1} =
    m(r-1)+s}}\hspace{-10mm}
E_i^{(s_1)}
\cdots
E_i^{(s_{m+1})},\label{epic2}\\
\Bigg[E_i^{(r)}, \hspace{-6mm}\sum_{\substack{s_1,\dots,s_m \geq r, t \geq 0 \\ s_1+\cdots+s_m+t =
    m(r-1)+s}}\hspace{-10mm}
D_i^{(t)} E_i^{(s_1)}
\cdots
E_i^{(s_m)}\Bigg]
&= (m-1)\hspace{-12mm}
\sum_{\substack{s_1,\dots,s_{m+1} \geq r, t \geq 0 \\ s_1+\cdots+s_{m+1}+t =
    (m+1)(r-1)+s}}
\hspace{-14mm} D_i^{(t)} E_i^{(s_1)}
\cdots
E_i^{(s_{m+1})},\label{epic3}
\end{align}\begin{align}
\Bigg[E_i^{(r)}, \hspace{-6mm}\sum_{\substack{s_1,\dots,s_m \geq r, t \geq 0 \\ s_1+\cdots+s_m+t =
    m(r-1)+s}}\hspace{-10mm}
D_{i+1}^{(t)} E_i^{(s_1)}
\cdots
E_i^{(s_m)}\Bigg]
&= (m+1)
\hspace{-12mm}\sum_{\substack{s_1,\dots,s_{m+1} \geq r, t \geq 0 \\ s_1+\cdots+s_{m+1}+t =
    (m+1)(r-1)+s}}
\hspace{-14mm} D_{i+1}^{(t)} E_i^{(s_1)}
\cdots
E_i^{(s_{m+1})},\label{epic4}\\
\Bigg[E_i^{(r)}, \hspace{-8mm}\sum_{\substack{s_1,\dots,s_m \geq r,
    t\geq 0, u
      \geq 0 \\ s_1+\cdots+s_m+t+u =
    m(r-1)+s}}\hspace{-12mm}
D_{i+1}^{(t)} E_i^{(s_1)}
\cdots
E_i^{(s_m)}
\tilde D_i^{(u)}\Bigg]
&= (m+2)\hspace{-16mm}
\sum_{\substack{s_1,\dots,s_{m+1} \geq r, t\geq 0, u \geq 0 \\
    s_1+\cdots+s_{m+1} +t+u=
    (m+1)(r-1)+s}}
\hspace{-17mm} D_{i+1}^{(t)} E_i^{(s_1)}
\cdots
E_i^{(s_{m+1})}
\widetilde{D}_i^{(u)}.\label{epic5}
\end{align}
\end{Lemma}

\begin{proof}
For (\ref{epic2}), we note in the summation on the left
that $s_1,\dots,s_m \leq r-1$ and $s_1+\cdots+s_m = (m-1)(r-1)+s$
implies
that $r-1 \geq s_1,\dots,s_m \geq s > 0$, so the expression makes
sense.
By (\ref{r6}), we have that
$$
[E_i^{(r)}, E_i^{(s_k)}] = - \sum_{\substack{s_k', s_k'' \leq r-1\\ s_k'+s_k'' = s_k+r-1}} 
\!\!\!\!\!E_i^{(s_k')} E_i^{(s_k'')}
$$
for $0 < s_k \leq r-1$.
Using this and the Leibniz rule, we deduce that the left hand side of
(\ref{epic2}) equals
\begin{multline*}
-\sum_{k=1}^m
\sum_{\substack{s_1,\dots,s_m \leq r-1 \\ s_1+\cdots+s_m =
    (m-1)(r-1)+s}}
\sum_{\substack{s_k', s_k'' \leq r-1 \\ s_k'+s_k'' = s_k+r-1}} 
\!\!\!\!\!E_i^{(s_1)} \cdots E_i^{(s_{k-1})}E_i^{(s_k')} E_i^{(s_k'')}E_i^{(s_{k+1})} \cdots
E_i^{(s_m)}
\\=
-\sum_{k=1}^m
\sum_{\substack{s_1,\dots,s_{m+1} \leq r-1 \\ s_1+\cdots+s_{m+1} =
    m(r-1)+s}}
\!\!\!\!\!E_i^{(s_1)} \cdots E_i^{(s_{m+1})},
\end{multline*}
which gives the right hand side of (\ref{epic2}). The proof of
(\ref{epic1}) is similar.
Then (\ref{epic3})--(\ref{epic4}) follow from (\ref{r4})
and (\ref{epic1}) with
one more application of Leibniz.
Finally (\ref{epic5}) follows similarly from (\ref{epic4}) together
with
\begin{equation}\label{epic6}
\big[E_i^{(r)}, \tilde D_i^{(s)}\big] = 
\sum_{t=0}^{s-1} E_i^{(r+s-1-t)} \tilde D_i^{(t)}.
\end{equation}
This may be deduced from (\ref{prel5}) by dividing by $(u-v)$ then
equating coefficients in exactly the manner explained in the paragraph
following \cite[(5.23)]{BK1}.
\end{proof}

The relations in the next lemma involve the expressions
\begin{align}
D_{i\downarrow m}(u) &:= 
D_i(u) D_i(u-1)\cdots D_i(u-m+1),\\
D_{i \uparrow m}(u) &:= D_i(u) D_i(u+1) \cdots D_i(u+m-1) .
\end{align}
In view of (\ref{r2}), the order of the products on the right
hand sides here is irrelevant.

\begin{Lemma}
The following relations hold for all $m \geq 1$:
\begin{align}\label{Beq1}
(u-v) [D_{i\downarrow m}(u), E_i(v)] &=
m D_{i \downarrow m}(u)(E_i(v)-E_i(u)),\\
(u-v) [D_{i\uparrow m}(u), E_{i-1}(v)] &=
m D_{i \uparrow m}(u)(E_{i-1}(u)-E_{i-1}(v)).
\label{Beq2}
\end{align}
\end{Lemma}

\begin{proof}
For (\ref{Beq1}), we actually prove it in the following equivalent
form:
\begin{align}\label{Beq3}
(u-v-m) D_{i\downarrow m}(u)E_i(v)
&= (u-v) E_i(v)
D_{i\downarrow m}(u)
- m D_{i\downarrow m}(u)
 E_i(u).
\end{align}
This follows when $m=1$ from (\ref{prel1}).
To prove (\ref{Beq3}) in general, proceed by induction: given
(\ref{Beq3}) for some $m \geq 1$, multiply both sides on the left by
$(u-v-m-1) D_i(u-m)$ then simplify using the $m=1$ case already
proved plus (\ref{jon2})
to obtain the analogous formula with $m$ replaced by $m+1$.

The proof of (\ref{Beq2}) is similar. One actually shows equivalently that
\begin{equation}\label{Beq6}
(u-v+m) D_{i\uparrow m}(u) E_{i-1}(v) =
(u-v) E_{i-1}(v) D_{i\uparrow m}(u) + m D_{i\uparrow m}(u) E_{i-1}(u).
\end{equation}
This follows when $m=1$ from (\ref{prel2}), 
then the general case
follows by a similar induction to the previous paragraph.
\end{proof}

For our final relations, we let
\begin{equation}\label{hum}
H_i(u) = \sum_{r \geq 0} H_i^{(r)} u^{-r} := -D_{i+1}(u) \widetilde{D}_i(u)
\end{equation}
assuming $1 \leq i < n$.
In particular, $H_i^{(0)} = -1$.
It is straightforward to see from (\ref{stuff}) that $H_i^{(r+1)} \in
\F_{r} Y_n$ and
\begin{equation}\label{stuffy}
\gr_r H_i^{(r+1)} = e_{i,i} t^r-e_{i+1,i+1} t^{r}.
\end{equation}
Note also by Corollary~\ref{violin} that
\begin{equation}\label{stuffier}
H_i(u) E_i(u-1) = E_i(u+1) H_i(u).
\end{equation}

\begin{Lemma}\label{hlem}
The following relations hold in $Y_n[[u^{-1}, v^{-1}]]$:
\begin{align}
\label{h2}(u-v-1)[H_{i}(u), E_i(v)] &= 2 H_{i}(u)(E_i(u-1)-E_i(v)),\\
\label{h2b}(u-v+1)[H_{i}(u), E_i(v)] &= 2 (E_i(u+1)-E_i(v)) H_{i}(u),\\
\label{h1}(u-v)[H_{i-1}(u), E_i(v)] &= 
- H_{i-1}(u)(E_i(u)-E_i(v)),\\
\label{h1b}
(u-v-1)[H_{i-1}(u), E_i(v)] &= 
- (E_i(u-1)-E_i(v)) H_{i-1}(u),\\
\label{h3b}(u-v)[H_{i+1}(u), E_i(v)] &= - (E_i(u)-E_i(v))
H_{i+1}(u),\\
\label{h3}(u-v+1)[H_{i+1}(u), E_i(v)] &= - H_{i+1}(u)(E_i(u+1)-E_i(v)).
\end{align}
\end{Lemma}

\begin{proof}
We just go through (\ref{h1}), (\ref{h3}) and (\ref{h2}), since
the other three are similar.
The identity (\ref{h1}) follows easily from (\ref{prel1}) and Leibniz, using that
$\widetilde{D}_{i-1}(v)$ commutes with $E_i(u)$ by (\ref{r2}).
For (\ref{h3}), we have by (\ref{jons}) that
$$
(u-v-1) E_i(u) \widetilde{D}_{i+1}(v) = (u-v) \widetilde{D}_{i+1}(v) E_i(u) -E_i(v)\widetilde{D}_{i+1}(v).
$$
Using also (\ref{jons3}) gives that
$$
(u-v-1) [E_i(u), \widetilde{D}_{i+1}(v)] = \widetilde{D}_{i+1}(v) E_i(u) -\widetilde{D}_{i+1}(v)
E_i(v+1).
$$
Now multiply by $D_{i+2}(v)$ (which commutes with $E_i(u)$) then
interchange $u$ and $v$ to get
the conclusion.
Finally, for (\ref{h2}), we have by (\ref{new4}) that
\begin{equation}\label{song}
(u-v)[E_i(u), D_{i+1}(v) \widetilde{D}_i(v)] = 2
D_{i+1}(v)(E_i(v)-E_i(u))\widetilde{D}_i(v).
\end{equation}
From (\ref{prel5}), we have that
$$
(u-v+1) E_i(u)  \widetilde{D}_i(v) = (u-v) \widetilde{D}_i(v) E_i(u) +
\widetilde{D}_i(v) E_i(v-1).
$$
Then we multiply (\ref{song}) by $(u-v+1)$ and use these identities
plus (\ref{jon2}) to obtain
$$
(u-v)(u-v+1)[E_i(u), H_i(v)] = 2 (u-v)
H_i(v)(
E_i(v-1)-E_i(u)).
$$
Dividing through by $(u-v)$ and interchanging $u$ and $v$ gives the result.
\end{proof}

\begin{Corollary}
The following hold in $Y_n[[u^{-1}, v^{-1}]]$:
\begin{align}
(u-v) [H_i(u), E_i(v)] &= - H_i(u)E_i(v)-E_i(v)H_i(u) + 2 H_i(u) E_i(u-1),\label{hcor0}\\
 (u-v)\Big[H_{i-1}\big(u+\half\big) , E_i(v) \Big] &= 
\half \Big(H_{i-1}(u+\half) E_i(v)
+
E_i(v) H_{i-1}\big(u+\half\big)\Big)\notag\\
&\hspace{-8mm}- \half \Big(H_{i-1}\big(u+\half\big) E_i\big(u+\half\big)+ E_i\big(u-\half\big) H_{i-1}\big(u+\half\big)\Big),\label{hcor1}\\
 (u-v)\Big[H_{i+1}\big(u-\half\big), E_i(v)\Big] &= 
\half \Big(H_{i+1}(u-\half) E_i(v)
+
E_i(v) H_{i+1}\big(u-\half\big)\Big)\notag\\
&\hspace{-8mm}- \half \Big(H_{i+1}\big(u-\half\big) E_i\big(u+\half\big)+ E_i\big(u-\half\big) H_{i+1}\big(u-\half\big)\Big),\label{hcor3}
\end{align}
assuming that $\operatorname{char}\k \neq 2$ for the last two
(so that $\half$ makes sense).
\end{Corollary}

\begin{proof}
When $\operatorname{char}\k \neq 2$, these follow by averaging
the corresponding pairs identities from Lemma~\ref{hlem}, e.g. 
(\ref{hcor1}) is
$($(\ref{h1})+(\ref{h1b})$)/2$ with $v$ replaced by $v+\half$.
For (\ref{hcor0}), one also needs to use (\ref{stuffier}).
To establish (\ref{hcor0}) when $\operatorname{char}\k=2$, 
we observe by (\ref{h2}) that $\big[H_i^{(r)}, E_i^{(s)}\big] = 0$ for all
$r,s > 0$, which easily implies the desired identity.
\end{proof}

\subsection{The modular shifted Yangian}
 
Pick a shift matrix $\sigma$ 
as
in \textsection \ref{shiftsec}. 
The \emph{shifted Yangian}
is the subalgebra $\Ys \subseteq Y_n$ generated by the following elements:
\begin{multline}\label{e:shiftedgens}
\big\{D_i^{(r)} \:\big|\: 1\leq i \leq n, r > 0\big\}
  \cup \big\{E_{i}^{(r)}, F_i^{(s)} \:\big |\: 1\leq i < n, r >
  s_{i,i+1}, s >
  s_{i+1,i}\big\}
\end{multline}
Notice that when $\sigma$ is the zero matrix we have $Y_n(\sigma) =
Y_n$. Neither the PBW basis nor the centre of $\Ys$ can be described without introducing higher root elements, 
however the elements $E_{i,j}^{(r)} \in Y_n$ do not lie in the
subalgebra $\Ys$ for
a general shift matrix.
Instead, following \cite[(2.18)--(2.19)]{BK2}, 
we recursively define 
\begin{align}\label{e:EcommShift}
\sE_{i,j}^{(r)} &:=\left[\sE_{i,j-1}^{(r - s_{j-1, j})},
  E_{j-1}^{(s_{j-1, j} + 1)}\right],
&\sF_{i,j}^{(s)} &:= \left[F_{j-1}^{(s_{j, j-1}+ 1)}, \sF_{i, j-1}^{(s - s_{j, j-1})}\right]
\end{align}
for $1\leq i < j \leq
n$
and $r > s_{i,j}, s > s_{j,i}$.
As always, the filtration on $Y_n$ induces a filtration on its subalgebra
$Y_n(\sigma)$ so that $\gr Y_n(\sigma) \hookrightarrow \gr Y_n$.

\begin{Lemma}\label{L:sigmathesame}
For any shift matrix $\sigma$,
$1\leq i < j \leq n$, and $r  \geq s_{i,j}, s\geq s_{j,i}$, we have that
$\sE_{i,j}^{(r+1)} \in \F_r Y_n$ and $\sF_{i,j}^{(s+1)} \in \F_{s} Y_n$. Moreover,
recalling that $\gr Y_n$ is identified with $U(\g)$,
we have that 
\begin{align}\label{rivers}
\gr_{r} \sE_{i,j}^{(r+1)} &= e_{i,j}t^{r},
&
\gr_{s} \sF_{i,j}^{(s+1)} &= e_{j,i}t^{s}.
\end{align}
Hence, 
$\gr Y_n(\sigma)$ is identified
with the subalgebra $\Us$ of $U(\g)$.
\end{Lemma}

\begin{proof}
We just prove the statements about $\sE_{i,j}^{(r+1)}$.
When $j=i+1$ the result is immediate from (\ref{stuff}).
To deduce it in general, 
we proceed by induction on $j-i$. For the induction step,
$\sE_{i,j}^{(r+1)} = \left[\sE_{i,j-1}^{(r+1 - s_{j-1, j})},
  E_{j-1}^{(s_{j-1, j} + 1)}\right]$ so by induction it lies in
$\F_{r-s_{j-1,j}+s_{j-1,j}} Y_n$ as required.
Moreover, by induction again, its image in the associated graded
algebra
is
$$
[e_{i,j-1}t^{r-s_{j-1,j}}, e_{j-1,j} t^{s_{j-1,j}}] = 
e_{i,j} t^r
$$
using (\ref{e:multloop}).
\end{proof}

From this and the PBW
theorem for $\Us$, we also get the following PBW theorem for $Y_n(\sigma)$:

\begin{Theorem}
Ordered monomials
in the elements 
\begin{equation}
\big\{D_i^{(r)} \:\big|\: 1\leq i \leq n, r > 0\big\} \cup
\big\{E_{i,j}^{(r)}, F_{i,j}^{(s)}\:\big|\: 1\leq i < j \leq n, r >
s_{i,j}, s > s_{j,i}\big\}
\end{equation}
taken in any fixed ordering 
form a basis for $Y_n(\sigma)$.
\label{L:PBWshifted}
\end{Theorem}

We have defined $Y_n(\sigma)$ as a subalgebra of $Y_n$. It can also be
defined by generators and relations: the following theorem shows that
it has its own Drinfeld presentation.

\begin{Theorem}\label{drinfeldshift}
The shifted Yangian $\Ys$ is generated by the elements
\eqref{e:shiftedgens} 
subject to the
relations
(\ref{r2})--(\ref{rel16}), interpreting ``admissible $i,j,r,s,t$'' so
that the left hand sides of these relations only involve generators of $\Ys$.
\end{Theorem}

\begin{proof}
Let $\widehat{Y}_n(\sigma)$ be the algebra defined by these generators
and relations. Since all of the relations hold in $Y_n(\sigma)$, there
is a homomorphism $\theta:\widehat{Y}_n(\sigma)\rightarrow
Y_n(\sigma)$
taking the generators to the elements of $Y_n(\sigma)$ with the same name.
Introduce higher root elements 
$\sE_{i,j}^{(r)}$ and $\sF_{i,j}^{(r)}$ in $\widehat{Y}_n(\sigma)$ by repeating
(\ref{e:EcommShift}), so that $\theta$ takes these to the elements
of $Y_n(\sigma)$ the same name too.
Then define a filtration on
$\widehat{Y}_n(\sigma)$ by placing $D_i^{(r)}, \sE_{i,j}^{(r)}$ and
$\sF_{i,j}^{(r)}$ 
into filtered degree $r-1$.
Arguing in the same way as the paragraph following (\ref{hydra}), 
we see that there is a surjective graded algebra homomorphism
\begin{equation*}
\psi:\Us \twoheadrightarrow \gr \widehat{Y}_n(\sigma),
\quad
e_{i,j} t^{r} \mapsto \left\{
\begin{array}{ll}
\gr_{r} D_i^{(r+1)}&\text{if $i=j$,}\\
\gr_{r} E_{i,j}^{(r+1)}&\text{if $i<j$,}\\
\gr_{r} F_{j,i}^{(r+1)}&\text{if $i>j$,}
\end{array}\right.
\end{equation*}
for $1 \leq i,j \leq n$ and $r \geq s_{i,j}$.
Finally, we observe that $\theta$ is filtered, and 
$$(\gr \theta) \circ \psi:\Us\rightarrow \gr Y_n(\sigma)
$$ 
is the identity by Lemma~\ref{L:sigmathesame}.
This implies that 
$\gr \theta$  is an isomorphism, hence, so is $\theta$.
\end{proof}

\subsection{Automorphisms}\label{saut}
In the next section, we will also need to exploit various
automorphisms/isomorphisms of Yangians and shifted Yangians.
We briefly list the ones that we need below. In all cases, existence
follows easily from the defining relations; see also
\cite[\textsection 2]{BK1}
and \cite[(2.16)--(2.17)]{BK2}.
\begin{enumerate}
\item(``Translation'') 
For $c \in \k$, there is an automorphism $\eta_c:Y_n \rightarrow Y_n$ defined from
$\eta_c(T_{i,j}(u)) = T_{i,j}(u-c)$, i.e. $\eta_c(T_{i,j}^{(r)}) =
\sum_{s=1}^{r} \binom{r-1}{r-s} c^{r-s} T_{i,j}^{(s)}$.
In terms of Drinfeld generators, $\eta_c$ sends $D_i(u) \mapsto D_i(u-c), E_{i,j}(u) \mapsto
E_{i,j}(u-c)$ and $F_{i,j}(u) \mapsto F_{i,j}(u-c)$, from which one
sees that $\eta_c$ does not leave $Y_n(\sigma)$ invariant in general.
\item(``Multiplication by a power series'')
For any power series $f(u) \in 1+u^{-1} \k[[u^{-1}]]$, there
is an automorphism $\mu_f:Y_n \rightarrow Y_n$
defined from $\mu_f(T_{i,j}(u)) = f(u) T_{i,j}(u)$,
i.e.
$\mu_f(T_{i,j}^{(r)}) = \sum_{s=0}^r a_s T_{i,j}^{(r-s)}$ if $f(u) =
\sum_{s \geq 0} a_s u^{-s}$.
On Drinfeld generators, we have that $\mu_f(D_i(u)) = f(u) D_i(u)$,
$\mu_f(E_{i}(u))  = E_i(u)$ and $\mu_f(F_i(u)) = F_i(u)$.
So this time $\mu_f$ restricts to an automorphism of each shifted
Yangian $Y_n(\sigma)$ fixing the off-diagonal generators (\ref{e:EcommShift}).
\item(``Transposition'')
Let $\tau:Y_n \rightarrow Y_n$ be the anti-automorphism
defined by $\tau(T_{i,j}^{(r)}) := T_{j,i}^{(r)}$ or, 
on Drinfeld generators, 
$\tau(D_i^{(r)}) = D_i^{(r)},
\tau(E_{i,j}^{(r)}) = F_{i,j}^{(r)},
\tau(F_{i,j}^{(r)}) = E_{i,j}^{(r)}$.
This restricts to define an anti-isomorphism $\tau:Y_n(\sigma) \rightarrow
Y_n(\sigma^T)$ between shifted Yangians,
where $\sigma^T$ is the transposed shift matrix.
\item(``Change of shift matrix'')
Suppose that $\sigma$ is a shift matrix as usual and 
$\dsigma = (\ds_{i,j})_{1\leq i,j\leq n}$ is another shift matrix satisfying
$s_{i,i+1} + s_{i+1, i} = \ds_{i,i+1} + \ds_{i+1, i}$
for all $1\leq i < n$. Then, as a consequence of Theorem~\ref{drinfeldshift},
there is an isomorphism $\iota:\Ys \isoto Y_n(\dsigma)$,
$D_i^{(r)} \mapsto D_i^{(r)}$, $\sE_{i,j}^{(r)} \mapsto
\dsE_{i,j}^{(r - s_{i,j} + \ds_{i,j})}$,
$\sF_{i,j}^{(r)} \mapsto \dsF_{i,j}^{(r - s_{j,i} + \ds_{j,i})}$.
\item(``Permutation'')
For each $w \in \S_n$, there is an automorphism
$w:Y_n \rightarrow Y_n$ sending $T_{i,j}^{(r)} \mapsto T_{w(i),
  w(j)}^{(r)}$. This is clear from the RTT presentation.
It does not leave the subalgebra $Y_n(\sigma)$ invariant in general.
\end{enumerate}

\begin{Lemma}\label{rl1}
For any $1 \leq i < j \leq n$, the permutation automorphism of $Y_n$
defined by the transposition $(i+1\:\:j)$ 
maps $E_i(u) \mapsto E_{i,j}(u)$ and $F_i(u) \mapsto F_{i,j}(u)$.
\end{Lemma}

\begin{proof}
Since $E_i(u) = E_{i,i+1}(u)$ and $F_i(u) = F_{i,i+1}(u)$,
this follows
from (\ref{qd1})--(\ref{qd3}).
\end{proof}

\section{Centres of $Y_n$ and $\Ys$}\label{S:thecentre}

In this section, 
we describe the centre of the modular shifted Yangian, giving precise formulas for the generators.
Unlike the previous section, most of the results presented here are
not analogues of statements regarding the shifted Yangian defined over
the complex numbers.

\subsection{Harish--Chandra centre}\label{s:HCinv}
We define a power series by the rule
\begin{eqnarray}\label{Cdef}
C(u) = \sum_{r\geq 0} C^{(r)} u^{-r} := 
D_1(u)D_2(u-1)\cdots D_n(u-n+1) \in Y_n[[u^{-1}]].
\end{eqnarray}
The algebra generated by the coefficients $\big\{C^{(r)} \:\big|\: r> 0\big\}$
will be denoted $Z_\HC(Y_n)$. We call it the {\it Harish-Chandra centre} of $Y_n$.
The following theorem shows that the associated graded algebra 
$\gr Z_{\HC}(Y_n)$ is identified with
$\k[z_r\:|\:r
\geq 0] \subseteq Z(\g)$.

\begin{Theorem}\label{HCgenerators}
The elements $C^{(r)}$ lie in the centre of $Y_n$. 
Furthermore, we have that $C^{(r+1)} \in \F_{r} Y_n$ and
\begin{equation}
\label{e:ctops}
\gr_r C^{(r+1)}  = z_r\in U(\g).
\end{equation}
Hence, $C^{(1)}, C^{(2)}, C^{(3)}, \dots$ are algebraically independent.
\end{Theorem}

\begin{proof}
To prove that $C^{(r)}$ is central, using (\ref{r2}) and the
anti-automorphism $\tau$ 
from \textsection \ref{saut}, we are reduced to checking that
$[E_i^{(r)}, C^{(s)}] = 0$ for all $i, r, s$. 
This can be proven in the same manner as \cite[Theorem~7.2]{BK1} using
the power series relations
(\ref{prel1}) and (\ref{prel2}). 
The second claim is noted in the proof of \cite[Theorem~7.2]{BK1}.
It can also be deduced using Lemma~\ref{L:type2},
taking $A_{\rm II} := Y_n$ and
$X_i^{(r)} := D_i^{(r)}$ so that $d_r = r-1$ for $r > 0$, and noting that
every $r > 1$ is optimal in this situation 
by Lemma~\ref{L:optlemma} with $m = 1$.
The final assertion follows because $z_0,z_1,\dots$ are algebraically
independent in $\gr Y_n$.
\end{proof}

Notice also that $Z_\HC(Y_n) \subseteq Z(\Ys)$ for any choice of shift matrix
and so we may also denote it $Z_\HC(\Ys)$ and call it the {\it
  Harish-Chandra centre} of $\Ys$.

\subsection{\boldmath Off-diagonal $p$-central elements}
We are ready to exhibit our first $p$-central elements.
In this subsection, we investigate the ones that lie in the ``root subalgebras''
$Y^\pm_{i,j} \subset Y_n$
for $1 \leq i < j \leq n$, that is, the subalgebras generated by
$\big\{E_{i,j}^{(r)}\:\big|\:r > 0\big\}$
and
$\big\{F_{i,j}^{(r)}\:\big|\:r > 0\big\}$, respectively.
In fact, we will give two different expressions for central
elements in $Y^\pm_{i,j}$ in the first two lemmas. The first of these involving power series produces
more complicated central elements, but this form often seems to be useful in practice.

\begin{Lemma}\label{heat}
For $1 \leq i< j \leq n$, all
coefficients in the power series $(E_{i,j}(u))^p$ and $(F_{i,j}(u))^p$
belong to $Z(Y_n)$.
\end{Lemma}

\begin{proof}
Using Lemma~\ref{rl1} plus the anti-automorphism
$\tau$, the proof of this reduces
to checking that the coefficients of $(E_i(u))^p$ are central in $Y_n$
for each $i=1,\dots,n-1$. To prove this, 
using Lemma~\ref{hop}, it suffices to establish
the following identities.
In $Y_n[[u^{-1}, v^{-1}]]$ for all admissible $j$:
\begin{align}
(\ad\,E_i(u))^p (E_j(v)) &= 0,\label{uo1}\\
(\ad\,E_i(u))^p(D_j(v)) &=0,\label{uo2}\\
(\ad\,E_i(u))^p(F_j(v)) &= 0.\label{uo3}
\end{align}

In this paragraph we check (\ref{uo1}).
To show that $[E_i(u)^p, E_i(v)] = 0$, we use
(\ref{prel4}) and (\ref{new1}) repeatedly:
\begin{align*}
(u-v)^p(\ad\, E_i(u))^{p-1}([E_i(u), E_i(v)])
&=
(u-v)^{p-1} (\ad\, E_i(u))^{p-1} ((E_i(v)-E_i(u))^2)\\
&=(u-v)^{p-2} (\ad\, E_i(u))^{p-2} (2(E_i(v)-E_i(u))^3)\\
&= \cdots = p! (E_i(v)-E_i(u))^{p+1} = 0.
\end{align*}
Dividing by $(u-v)^p$ (which we may do since
$Y_n[[u^{-1}, v^{-1}]][u,v]$ has no zero divisors) gives the desired identity.
To see that $(\ad\, E_i(u))^p (E_j(v)) = 0$ when $|i-j|=1$, 
we actually already have that $(\ad\,E_i(u))^2(E_j(v)) = 0$ by (\ref{star}).
Finally, when $|i-j|> 1$, the identity is clear because
$[E_i(u), E_j(v)] = 0$ by (\ref{r10}).

For (\ref{uo2}), it is immediate from (\ref{r4}) if $j < i$ or $j >
i+1$.
For the case $j=i$, we have by (\ref{prel1}) and (\ref{new2}) with
$k=1$ that $(u-v)^2 (\ad\,E_i(u))^2 (D_i(v)) = 0$.
Hence, on dividing by $(u-v)^2$, we get that
$(\ad\,E_i(u))^p (D_i(v)) = 0$.
Finally, when $j=i+1$, we use (\ref{prel2}) and (\ref{new3})
repeatedly:
\begin{align*}
(u-v)^p(\ad\, E_i(u))^{p}(D_{i+1}(v))
&=
(u-v)^{p-1} (\ad\, E_i(u))^{p-1} (D_{i+1}(v)(E_i(v)-E_i(u)))\\
&=(u-v)^{p-2} (\ad\, E_i(u))^{p-2} (2(E_i(v)-E_i(u))^2)\\
&= \cdots = p! (E_i(v)-E_i(u))^{p} = 0.
\end{align*}
Dividing by $(u-v)^p$ completes the proof of (\ref{uo2}).

Finally, for (\ref{uo3}), it follows when $i \neq j$ 
immediately from (\ref{r3}). When $i=j$, 
we observe using (\ref{new4}) repeatedly that
\begin{align*}
(u-v)^{p-1}(\ad\, E_i(u))^{p-1}&(D_{i+1}(v)\widetilde D_i(v))\\&=
(u-v)^{p-2}(\ad\, E_i(u))^{p-2}(2 D_{i+1}(v)(E_i(v)-E_i(u))\widetilde
D_i(v)) \\
&=\cdots=
p! D_{i+1}(v)(E_i(v)-E_i(u))^{p-1}\widetilde
D_i(v) = 0.
\end{align*}
Hence, 
$(\ad\, E_i(u))^{p-1}(D_{i+1}(v) \widetilde D_i(v)) = 0$. We
can also set $v = u$ in this identity to see that
$(\ad\, E_i(u))^{p-1}(D_{i+1}(u) \widetilde D_i(u)) = 0$.
Then using  
(\ref{prel3}) we get the conclusion:
\begin{align*}
(u-v) (\ad\, E_i(u))^{p}(F_i(v)) = 
 (\ad\, E_i(u))^{p-1}(D_{i+1}(u) \widetilde D_i(u) -
D_{i+1}(v)\widetilde D_i(v)) = 0.
\end{align*}
\end{proof}

\begin{Lemma}\label{alem}
For $1 \leq i < j \leq n$ and $r > 0$, 
we have that $\big(E_{i,j}^{(r)}\big)^p, \big(F_{i,j}^{(r)}\big)^p \in Z(Y_n)$.
\end{Lemma}

\begin{proof}
Using Lemmas~\ref{rl1} and \ref{hop}, this reduces to checking 
\begin{align}
\big(\ad\, E_i^{(r)}\big)^p\left(E_j^{(s)}\right)&= 0,\label{a1}\\
\big(\ad\, E_i^{(r)}\big)^p\left(D_j^{(s)}\right) &= 0,\label{a2}\\
\big(\ad\, E_i^{(r)}\big)^p\left(F_j^{(s)}\right) &= 0.\label{a3}
\end{align}
These may all be proved in a very similar way to
(\ref{uo1})--(\ref{uo3}), using the identities from Lemma~\ref{lf}
in place of the ones from Lemma~\ref{le} used before, and also
(\ref{r3}) and (\ref{rel15}).
For example, to prove (\ref{a1}) when $i=j$, we use (\ref{epic1}) if
$s \geq r$ or (\ref{epic2}) if $s \leq r-1$,
taking $m=1,2,\dots,p$ in turn. Here is the calculation in the latter
case:
\begin{align*}
\big(\ad\, E_i^{(r)}\big)^p\left(E_i^{(s)}\right)
&=
-\big(\ad\, E_i^{(r)}\big)^{p-1}\left(\sum_{\substack{s_1,s_2 \leq r-1 \\ s_1+s_2 =
    r-1+s}}
E_i^{(s_1)}E_i^{(s_2)}\right)\\
&=2\big(\ad\, E_i^{(r)}\big)^{p-2}\left(\sum_{\substack{s_1,s_2,s_3 \leq r-1 \\ s_1+s_2+s_3 =
    2(r-1)+s}}
E_i^{(s_1)}E_i^{(s_2)}E_i^{(s_3)}\right)\\
&=\cdots=
(-1)^p p! \sum_{\substack{s_1,\dots,s_{p+1} \leq r-1 \\ s_1+\cdots+s_{p+1} =
    p(r-1)+s}}
E_i^{(s_1)}\cdots E_i^{(s_{p+1})} = 0.
\end{align*}
\end{proof}

For use in the next theorem, we let
\begin{align}\label{pq}
P_{i,j}(u) &= \sum_{r \geq p} P_{i,j}^{(r)}u^{-r} := E_{i,j}(u)^p,&
Q_{i,j}(u) &= \sum_{r \geq p} Q_{i,j}^{(r)}u^{-r} := F_{i,j}(u)^p.
\end{align}

\begin{Theorem}\label{Z1}
For $1 \leq i < j \leq n$, the algebras
$Z(Y_n) \cap Y_{i,j}^+$ and
$Z(Y_n) \cap Y_{i,j}^-$
are infinite rank polynomial algebras freely generated by the central elements
$\Big\{\big(E_{i,j}^{(r)}\big)^p\:\Big|\:r > 0\Big\}$ and
$\Big\{\big(F_{i,j}^{(r)}\big)^p\:\Big|\:r > 0 \Big\}$, respectively.
We have that
$\big(E_{i,j}^{(r)}\big)^p,
\big(F_{i,j}^{(r)}\big)^p
 \in \F_{rp-p} Y_n$
and 
\begin{equation}\label{snore}
\gr_{rp-p} 
\big(E_{i,j}^{(r)}\big)^p
= \big(e_{i,j} t^{r-1}\big)^p,
\qquad
\gr_{rp-p} 
\big(F_{i,j}^{(r)}\big)^p
= \big(e_{j,i} t^{r-1}\big)^p.
\end{equation}
For $r \geq p$ we have
that
\begin{equation}\label{kaz}
P_{i,j}^{(r)} = \left\{
\begin{array}{ll}
\left(E_{i,j}^{(r/p)}\right)^p+(*)&\text{if $p \mid r$,}\\
(*)&\text{if $p \nmid r$,}
\end{array}
\right.
\end{equation}
where $(*)\in
 \F_{r-p-1} Y_n$
is a polynomial in
the elements
$\left(E_{i,j}^{(s)}\right)^p$ for $1 \leq s < \lfloor r/p\rfloor$.
Hence, the central elements
$\big\{P_{i,j}^{(rp)}\:\big|\:r > 0\big\}$ give another algebraically
independent set of generators for 
$Z(Y_n) \cap Y_{i,j}^+$ 
lifting the central elements $\big\{\left(e_{i,j}t^{r-1}\right)^p\:\big|\:r >
0\big\}$ of $\gr Y_n$.
Analogous statements with $Y_{i,j}^+, E, P$ and $e_{i,j}t^{r-1}$ 
replaced by $Y_{i,j}^-,F,Q$ and $e_{j,i}t^{r-1}$ also hold.
\end{Theorem}

\begin{proof}
It suffices to prove all of the statements for $Y_{i,j}^+$; then they
follow for $Y_{i,j}^-$ using the anti-automorphism $\tau$ from
\textsection \ref{saut}.
Let $\g^+_{i,j}$ be the commutative subalgebra of $\g$ spanned by $\big\{
e_{i,j}t^r\:\big|\:r \geq 0\big\}$.
By Theorem~\ref{loopcentre}, it is easy to see that
\begin{equation}\label{finalz1}
Z(\g) \cap U(\g^+_{i,j}) = \k\Big[\big(e_{i,j} t^r\big)^p\:\Big|\:r \geq
0\Big].
\end{equation}
Note that $\gr (Z(Y_n) \cap Y^+_{i,j}) \subseteq Z(\g)\cap U(\g^+_{i,j})$.
By Lemma~\ref{alem},
we know that
$\big(E_{i,j}^{(r)}\big)^p$ belongs to $Z(Y_n) \cap Y^+_{i,j}$.
Moreover, by (\ref{stuff}), it is clear that
$\big(E_{i,j}^{(r)}\big)^p$
is of filtered degree $rp-p$ with
$\gr_{rp-p}  \big(E_{i,j}^{(r)}\big)^p = \left(e_{i,j}t^{r-1}\right)^p$. 
It follows that $\gr (Z(Y_n)\cap Y^+_{i,j})= Z(\g)\cap
U(\g^+_{i,j})$
and the elements
$\Big\{\big(E_{i,j}^{(r)}\big)^p\:\Big|\:r > 0\Big\}$ are
algebraically independent generators for $Z(Y_n)\cap Y^+_{i,j}$.

Now consider (\ref{kaz}).
By Lemma~\ref{heat}, each $P_{i,j}^{(r)}$ belongs to $Z(Y_n) \cap Y^+_{i,j}$,
hence, it is a polynomial in 
the elements
$\Big\{\big(E_{i,j}^{(s)}\big)^p\:\Big|\:s > 0\Big\}$.
It remains to show that $P_{i,j}^{(r)} \in \F_{r-p-1} Y_n$
if $p \nmid r$, or that $P_{i,j}^{(r)} \equiv
\Big(E_{i,j}^{(r/p)}\Big)^p \pmod{\F_{r-p} Y_n}$
if $r \mid p$.
This follows by an application of Lemma~\ref{I}, taking $A_{\rm I} := Y^+_{i,j}$
and $X^{(r)} := E_{i,j}^{(r)}$
so that $X_{\rm I}(u) = E_{i,j}(u)$.
\end{proof}

\begin{Corollary}\label{rockcor}
For any $\sigma$,
the elements
$\Big\{\big(\sE_{i,j}^{(r)}\big)^p\:\Big|\:r > s_{i,j}\Big\}$ 
and $\Big\{\big(\sF_{i,j}^{(r)}\big)^p\:\Big|\:r > s_{j,i}\Big\}$ 
are central in $Y_n(\sigma)$.
Moreover,
they
belong to
$\F_{rp-p} Y_n(\sigma)$
and 
\begin{equation}\label{rock}
\gr_{rp-p} 
\big(\sE_{i,j}^{(r)}\big)^p
= \big(e_{i,j} t^{r-1}\big)^p,
\qquad
\gr_{rp-p} 
\big(\sF_{i,j}^{(r)}\big)^p
= \big(e_{j,i} t^{r-1}\big)^p.
\end{equation}
\end{Corollary}

\begin{proof}
The assertion (\ref{rock}) is immediate from (\ref{rivers}), so we just
need to establish the centrality. In case $\sigma$ is upper
triangular,
we have that $\sE_{i,j}^{(r)} = E_{i,j}^{(r)}$, which is central in
$Y_n$ by Lemma~\ref{alem} so certainly central in the subalgebra $Y_n(\sigma)$.
The centrality of $\sE_{i,j}^{(r)}$ in general then follows using
the change of shift matrix isomorphism from \textsection \ref{saut}.
The centrality of $\sF_{i,j}^{(r)}$ is proved similarly.
\end{proof}

\begin{Remark}\label{later1}
The algebra $Y^+_{i,i+1}$ is generated by $\big\{E_i^{(r)}\:\big|\:r >
0\big\}$
subject just to the relations (\ref{r6}); these give enough relations
because they suffice to establish that the ordered monomials in the
generators span $Y_{i,i+1}^+$.
Thus, 
we are in the situation of Remark~\ref{bladerunner} with $X^{(r)} := E_i^{(r)} \in \F_{r-1} Y^+_{i,j}$.
Theorem~\ref{Z1} shows that $P_{i,i+1}^{(r)}$ is a polynomial in
$\Big\{\big(E_{i}^{(s)}\big)^p\:\big|\:0 < s \leq \lfloor
r/p\rfloor\Big\}$.
This establishes the claim made in Remark~\ref{bladerunner}.
\end{Remark}

\subsection{\boldmath Diagonal $p$-central elements}
Next we introduce the $p$-central elements
that belong to the diagonal subalgebras
\begin{equation}
Y_i^0 := \k\big[D_i^{(r)}\:\big|\:r > 0\big]
\end{equation}
of $Y_n$. Note $Y_i^0$ is also a subalgebra of
$Y_n(\sigma)$ for any shift matrix $\sigma$.
For $i=1,\dots,n$, we define
\begin{eqnarray}\label{definitionofB}
B_i(u) =\sum_{r \geq 0} B_i^{(r)}u^{-r} := 
D_i(u)D_i(u-1)\cdots D_i(u-p+1)
\in Y_i^0[[u^{-1}]].
\end{eqnarray}

\begin{Lemma}\label{walk}
For all $i=1,\dots,n$ and $r > 0$, 
the element $B_i^{(r)}$ belongs to $Z(Y_n)$. 
\end{Lemma}

\begin{proof}
In view of (\ref{r2}), it suffices to check 
for $1 \leq j < n$
that
$$
[B_i(u), E_j(v)] =0 = [B_i(u), F_j(v)]
$$
in $Y_n[[u^{-1}, v^{-1}]]$.
By applying the anti-automorphism $\tau$ from \textsection \ref{saut},
it actually suffices to
check just the first equality.
This is clear when $j \notin \{i-1, i\}$ by (\ref{r4}).
When $j=i-1$ or $j=i$, 
it follows
from the identities
(\ref{Beq1})--(\ref{Beq2})
taking $m:=p$, noting that
$B_i(u) = D_{i\downarrow p}(u) = D_{i\uparrow p}(u)$ by (\ref{r2}).
\end{proof}

\begin{Theorem}\label{dc}
Assume that $n \geq 2$.
For $1 \leq i \leq n$,
the algebra $Z(Y_n) \cap Y_i^0$
is an infinite rank polynomial algebra freely generated by
the central elements $\big\{B_i^{(rp)}\:\big|\:r > 0\big\}$.
This statement also
describes the algebras $Z(Y_n(\sigma))\cap Y_i^0$ for any shift
matrix $\sigma$.
We have that $B_i^{(rp)} \in \F_{rp-p} Y_n$
and
\begin{equation}\label{creek}
\gr_{rp-p} B_i^{(rp)} = (e_{i,i} t^{r-1})^p - e_{i,i} t^{rp-p}.
\end{equation}
For $0 < r < p$, we have that
$B_i^{(r)} = 0$. For $r \geq p$ with $p \nmid
r$, we have that $B_i^{(r)}\in \F_{r-p-1} Y_n$ and it
is a polynomial in the elements $\big\{B_i^{(sp)}\:\big|\:0 < s \leq
\lfloor r/p \rfloor\big\}$.
\end{Theorem}

\begin{proof}
Let $\g^0_{i}$ be the abelian subalgebra of $\g$ spanned by $\big\{
e_{i,i}t^r\:\big|\:r \geq 0\big\}$.
By Theorem~\ref{loopcentre} and the assumption $n \geq 2$,
one sees that
\begin{equation}\label{finalz2}
Z(\g) \cap U(\g^0_{i}) = \k\Big[\big(e_{i,i} t^r\big)^p -
e_{i,i}t^{rp} \:\Big|\:r \geq
0\Big].
\end{equation}
We have that 
$\gr (Z(Y_n) \cap Y_{i}^0) \subseteq Z(\g)\cap U(\g_{i}^0)$.
By Lemma~\ref{walk},
we know that
$B_i^{(rp+p)}$ belongs to $Z(Y_n) \cap Y_{i}^0$.
Moreover, applying Lemma~\ref{birds} with $A_{\rm III} = Y_i^0$ and
$X^{(r)} = D_i^{(r)}$, we see that 
$B_i^{(rp+p)} \in \F_{rp} Y_i^0$ and $\gr_{rp} B_i^{(rp+p)} = (e_{i,i}  t^{r})^p - e_{i,i}t^{rp}$.
We deduce that $\gr (Z(Y_n)\cap Y_{i}^0)= Z(\g)\cap
U(\g_{i}^0)$
and the elements
$\big\{B_i^{(rp)}\:\big|\:r > 0\big\}$ are
algebraically independent generators for it.
The same argument works in any $Y_n(\sigma)$.

Lemma~\ref{birds} also implies that $B_i^{(r)} = 0$ for $0 < r < p$
and that
$B_i^{(r)} \in F_{r-p-1} Y_i^0$ if $r \geq p$ with $p \nmid r$.
In this case, since it is central by Lemma~\ref{walk}, it must be a
polynomial in the elements
$\big\{B_i^{(sp)}\:\big|\:0 < s \leq
\lfloor r/p \rfloor\big\}$.
\end{proof}

\begin{Remark}\label{later2}
By Theorem~\ref{L:PBWshifted}, $Y_i^0$ is a free polynomial algebra
generated by the elements 
$D_i^{(1)}, D_i^{(2)},\dots$.
So Theorem~\ref{dc} also establishes the claim made in Remark~\ref{rain}.
\end{Remark}

\subsection{Main Theorem}
Now we can state and prove our main results.
Let $\sigma$ be any shift matrix.
We have already defined the Harish-Chandra centre
$Z_{\HC}(Y_n(\sigma))$ at the end of \textsection \ref{s:HCinv}.
Also define the {\em $p$-centre} $Z_p(Y_n(\sigma))$
of $Y_n(\sigma)$ to be the subalgebra generated by
\begin{equation}\label{pcent}
\big\{B_i^{(rp)}\:\big|\:1\leq i \leq n, r > 0\big\} \cup 
\Big\{\big(\sE_{i,j}^{(r)} \big)^p,
\big(\sF_{i,j}^{(r)} \big)^p\:\Big|\:1 \leq i < j \leq n, r > s_{i,j}, s > s_{j,i}\Big\}.
\end{equation}
We have shown that both $Z_{\HC}(Y_n(\sigma))$ and $Z_p(Y_n(\sigma))$
are subalgebras of $Z(Y_n(\sigma))$; see Corollary~\ref{rockcor} and
Lemma~\ref{walk}.
Note also by (\ref{rivers}) and (\ref{creek})
that $\gr Z_p(\Ys)$ may be identified with the $p$-centre $Z_p(\gs)$ of $\Us$ from (\ref{zdf}).

We also need one more family of elements: recalling (\ref{Cdef}) and
(\ref{definitionofB}), we let
\begin{equation}\label{BC}
BC(u) := \sum_{r\geq 0} BC^{(r)} u^{-r} := 
B_1(u) B_2(u-1)\cdots B_n(u-n+1) = C(u) C(u-1)\cdots C(u-p+1).
\end{equation}
From this definition, it follows that 
each $BC^{(r)}$ can be expressed as a polynomial in the elements
$\big\{B_i^{(s)}\:\big|\:1 \leq i \leq n, s > 0\big\}$,
so that it belongs to $Z_p(Y_n(\sigma))$ by Theorem~\ref{dc}.
Moreover, it is also a polynomial in the elements
$\big\{C^{(s)}\:\big|\:s > 0\big\}$, so that it belongs to
$Z_{\HC}(Y_n(\sigma))$.
We have just shown that
$BC^{(r)} \in Z_{\HC}(Y_n(\sigma))\cap Z_p(Y_n(\sigma))$.

\begin{Lemma}
For $r > 0$, 
we have that $BC^{(rp)} \in F_{rp-p} \Ys$ 
and
\begin{equation}\label{lifty}
\gr_{rp} BC^{(rp)}= z_{r-1}^p - z_{rp-p}.
\end{equation}
\end{Lemma}

\begin{proof}
Let $d_r := 0$ for $r < p$ and $d_r := p\lfloor r/p\rfloor-p$ for $r
\geq p$.
Theorem~\ref{dc} implies that $B_i^{(r)} \in \F_{d_r} Y_n(\sigma)$
for every $r > 0$.
Now apply Lemma~\ref{L:type2} 
with $A_{\rm II} = Y_n(\sigma)$ and $X_i^{(r)} =B_i^{(r)}$,
noting that $pr$ is optimal for every $r > 1$ by Lemma~\ref{L:optlemma}, to 
deduce that $BC^{(rp)} \in F_{rp-p} Y_n(\sigma)$
and
$$
BC^{(rp)} \equiv B_1^{(rp)}+\cdots+B_n^{(rp)} \pmod{F_{rp-p-1}
Y_n(\sigma)}.
$$
We are now done thanks to (\ref{creek}) once again.
\end{proof}

\begin{Theorem}\label{main}
The centre $Z(\Ys)$ is generated by $Z_{\HC}(\Ys)$ and
$Z_p(\Ys)$. Moreover:
\begin{enumerate}
\item $Z_{\HC}(\Ys)$ is the free polynomial algebra
generated by $\big\{C^{(r)}\:\big|\:r > 0\big\}$;
\item $Z_p(\Ys)$ is the free polynomial algebra
generated by 
\begin{equation}\label{polly}
\big\{B_i^{(rp)}\:\big|\:1\leq i\leq n, r > 0\big\} \cup 
\Big\{\big(\sE_{i,j}^{(r)} \big)^p,
\big(\sF_{i,j}^{(s)} \big)^p\:\Big|\:1 \leq i< j \leq n, r > s_{i,j}, s > s_{j,i}\Big\};
\end{equation}
 \item
$Z(\Ys)$ is the free polynomial algebra generated by
\begin{equation}
\label{mon0}
\big\{B_i^{(rp)}, C^{(r)}\:\big|\:2 \leq i \leq n, r > 0\big\} \cup 
\Big\{\big(\sE_{i,j}^{(r)} \big)^p,
\big(\sF_{i,j}^{(s)} \big)^p\:\Big|\:1 \leq i< j \leq n, r > s_{i,j}, s > s_{j,i}\Big\};
\end{equation}
\item
$Z_{\HC}(\Ys) \cap Z_p(\Ys)$ is the free polynomial algebra generated
by
$\big\{BC^{(rp)}\:\big|\:r > 0\big\}$.
\end{enumerate}
\end{Theorem}

\begin{proof}
(1) This is Theorem~\ref{HCgenerators}.

(2) The given elements generate $Z_p(\Ys)$ by the definition. We just
need to observe that they are algebraically independent. This follows
because by (\ref{rock}) and (\ref{creek}) they are lifts of the
algebraically independent generators of the $p$-centre of the
associated graded algebra from (\ref{zdf}).

(3)
Let $Z$ be the subalgebra of $Z(\Ys)$ generated by the given elements.
We have that
\begin{equation}\label{here}
\gr Z \subseteq \gr Z(\Ys) \subseteq 
Z(\gr \Ys) = Z(\gs).
\end{equation}
We have seen already that generators of $Z$
are
lifts of the algebraically independent generators of $Z(\gs)$
from (\ref{zgen}). 
Hence, they are algebraically independent and 
equality holds everywhere in (\ref{here}).
This implies that $Z = Z(\Ys)$.

(4)
We have already observed that all $BC^{(rp)}$
belong to $Z_\HC(\Ys)\cap Z_p(\Ys)$. Also they are algebraically
independent as
they are lifts of algebraically independent elements of $U(\gs)$ by
(\ref{lifty}).
We claim that $Z_p(\Ys)$ is freely generated by the elements
$$
\big\{BC^{(rp)}, B_i^{(rp)}\:\big|\:2\leq i\leq n, r > 0\big\} \cup 
\Big\{\big(\sE_{i,j}^{(r)} \big)^p,
\big(\sF_{i,j}^{(s)} \big)^p\:\Big|\:1 \leq i< j \leq n, r > s_{i,j}, s > s_{j,i}\Big\}.
$$
The result follows from the claim since we know already from (3) that
all of these elements different from $BC^{(rp)}$
are
algebraically independent of anything in $Z_\HC(\Ys)$.

To prove the claim, we use (\ref{rock}) (\ref{creek}) and
(\ref{lifty}) to pass to the associated graded algebra, thereby
reducing to showing that
$$
\big\{z_r^p-z_{rp}, (e_{i,i}t^r)^p - e_{i,i}t^{rp}\:\big|\:2\leq i\leq n, r \geq 0\big\} \cup 
\Big\{\big(e_{i,j}t^r\big)^p\:\Big|\:1 \leq i\neq j \leq n, r \geq
s_{i,j}\Big\}
$$
freely generate $Z_p(\gs)$.
This is easily seen by comparing them to the algebraically independent
generators from (\ref{zdf}).
\end{proof}

\begin{Corollary}\label{nearly}
The shifted Yangian $\Ys$ is free as a module over its centre, with
basis given by the ordered monomials in 
\begin{equation}\label{mon1}
\big\{D_i^{(r)}\:\big|\:2 \leq i \leq n, r >
0\big\}\cup\big\{\sE_{i,j}^{(r)},
\sF_{i,j}^{(s)}\:\big|\:1\leq i < j \leq n, r > s_{i,j}, s
> s_{j,i}\big\}
\end{equation}
in which no exponent is $p$ or more.\end{Corollary}

\begin{proof}
It suffices to show that the set consisting 
of 
ordered monomials in (\ref{mon0}) multiplied by
ordered monomials in (\ref{mon1}) with all
exponents $< p$
gives a basis for $\Ys$.
To see this, we pass to the associated graded algebra using
(\ref{stuff}), (\ref{rivers}), (\ref{e:ctops}), (\ref{rock}) and
(\ref{creek}) 
to reduce to showing that the monomials
$$
\prod_{r \geq 0}
z_r^{a_{1,1,r}}
\prod_{\substack{2 \leq i\leq n\\r \geq 0}}
\big((e_{i,i} t^r)^p - e_{i,i}t^{rp}\big)^{a_{i,i;r}}\!\!
\prod_{\substack{1 \leq i \neq j \leq n \\ r > s_{i,j}}}
(e_{i,j}t^r)^{pa_{i,j,r}}
\prod_{\substack{2\leq i \leq n\\r \geq 0}}
(e_{i,i} t^r)^{b_{i,i;r}}
\prod_{\substack{1 \leq i \neq j \leq n \\ r > s_{i,j}}}
(e_{i,j}t^r)^{b_{i,j,r}}
$$
for $a_{i,j;r}\geq 0$ and $0 \leq b_{i,j,r} < p$
form a basis for $U(\gs)$.
This is quite straightforward: these monomials are related to
a PBW basis by a uni-triangular transition matrix.
\end{proof}

\begin{Corollary}\label{nearly2}
The shifted Yangian $\Ys$ is free as a module over $Z_p(\Ys)$ with basis
given by the ordered monomials in
$$
\big\{D_i^{(r)}\:\big|\:1 \leq i \leq n, r >
0\big\}\cup\big\{\sE_{i,j}^{(r)},
\sF_{i,j}^{(s)}\:\big|\:1\leq i < j \leq n, r > s_{i,j}, s
> s_{j,i}\big\}
$$
in which no exponent is $p$ or more.
\end{Corollary}

\begin{proof}
Similar to the previous corollary.
\end{proof}

\begin{Remark}\label{two}
Assume in this remark that $n=2$ and $p=2$.
Recalling (\ref{hum}), we denote $E_1^{(r)}, F_1^{(r)}$ and $H_1^{(r)}$ simply by $E^{(r)},
F^{(r)}$ and $H^{(r)}$. 
Since $H(u) = -B_2(u) C(u)^{-1}$, the elements $H^{(r)}$ are all
central; this also follows from (\ref{h2}).
Consider $\overline{Y}_2 := Y_2 / I$ where $I$ is the two-sided ideal
generated by the central elements $(E^{(r)})^2$ and $(F^{(r)})^2$
for all $r > 0$.
An induction exercise using (\ref{r6})--(\ref{r7}) shows that the
following relations hold in $\overline{Y}_2$:
$[E^{(r)}, E^{(s)}] = [F^{(r)}, F^{(s)}] = 0$
for $r, s > 0$. 
Comparing with the presentation obtained in \cite[Theorem 3]{Gow},
we see that $\overline{Y}_2$ may be identified with the Yangian
of the Lie superalgebra $\mathfrak{gl}_{1|1}$ in characteristic 2.
\end{Remark}

\section{Modular Yangian of $\sl_n$}

In this section, we define a subalgebra $SY_n < Y_n$ which we call the
{\em special Yangian}, and show that this may be viewed as the modular
version of the Yangian for the Lie algebra $\sl_n$ rather than
$\gl_n$.
Then we use this connection to establish the results about $Z(Y_n)$ formulated in
terms of the RTT generators in
the introduction.

\subsection{The special Yangian}
We would like to mimic Drinfeld's
definition of $SY_n$ in characteristic zero from \cite{MNO};
see \cite[\textsection 2.24]{MNO} for its history.
Unfortunately, the
approach in {\em loc. cit.}
only works verbatim when the ground field is infinite.
Rather than insisting on that here, we will modify the definition
slightly by incorporating 
base change.
For any field extension $\K\supseteq \k$, we can define the Yangian
over $\K$ by  generators and relations in the same way as $Y_n$ was
defined over $\k$. The resulting $\K$-algebra
may be identified with $Y_n \otimes
\K$ in an obvious way.
Then the automorphisms $\mu_f$ defined as in \textsection \ref{saut} can be
viewed as
$\K$-linear
automorphisms of $Y_n\otimes \K$ for all $f(u) \in 1 + u^{-1}\K[[u^{-1}]]$.
Define the {\em special Yangian} to be
\begin{equation}\label{poppy}
SY_n := \left\{x \in Y_n\:\Bigg|\:
\begin{array}{ll}
\mu_f(x \otimes 1) = x\otimes 1\text{
    in }Y_n\otimes \K\\\text{for all }f(u) \in 1 +
u^{-1} \K[[u^{-1}]]\\
\text{and all field extensions }\K \supseteq\k
\end{array}
\right\}.
\end{equation}
In fact, as is clear from the proof of the next theorem, it
is enough just to take one infinite field $\K$ here.

Our first task is to identify the associated graded algebra $\gr SY_n
\subseteq \gr Y_n = U(\g)$ with
$U(\g')$, where $\g' := \sl_n[t]$ is the current algebra associated to $\sl_n$.
Recall the elements $H_i^{(r)}$ from (\ref{hum}).

\begin{Theorem}\label{oldsy}
The algebra $SY_n$ has a basis consisting of ordered monomials in
\begin{equation}\label{beans}
\big\{H_i^{(r)}\:\big|\:1\leq i <n, r > 0\big\}\cup
\big\{E_{i,j}^{(r)}, F_{i,j}^{(r)}\:\big|\:1 \leq i< j \leq n, r
> 0\big\}
\end{equation}
taken in any fixed order.
Hence, $\gr SY_n = U(\g')$,
and multiplication defines a {\em vector space} isomorphism
\begin{equation}\label{beans1}
SY_n \otimes Y_1 \stackrel{\sim}{\rightarrow} Y_n
\end{equation}
where $Y_1$ is identified with the subalgebra of $Y_n$ generated by
the elements 
$D_1^{(r)}$ in the obvious way.
If we assume in addition that $p \nmid n$ then multiplication defines an {\em algebra} isomorphism
\begin{equation}\label{beans2}
SY_n \otimes Z_{\HC}(Y_n) \stackrel{\sim}{\rightarrow} Y_n.
\end{equation}
\end{Theorem}

\begin{proof}
For any $\K\supseteq \k$ and $f(u)\in 1+u^{-1}\K[[u^{-1}]]$, 
we have that $\mu_f(D_{i}(u) \otimes 1) = f(u) D_{i}(u) \otimes 1$
by definition, hence,
$\mu_f(\widetilde{D}_{i}(u)\otimes 1) = f(u)^{-1}
\widetilde{D}_i(u)\otimes 1$.
It follows that
$H_{i}^{(r)} \in SY_n$,
and of course all $E_{i,j}^{(r)}$ and $F_{i,j}^{(r)}$ lie in $SY_n$
too.
Hence, ordered monomials in the elements (\ref{beans}) span a subspace
$\overline{SY}_n$ of $SY_n$. We are shortly going to
prove that $\overline{SY}_n = SY_n$, so that $\overline{SY}_n$ is a
subalgebra; one could also prove this right away using
Lemma~\ref{hlem} but actually we do not
need it for the proof below.

The filtration on $Y_n$ induces a vector space filtration on
$\overline{SY}_n$
so that $\gr \overline{SY}_n$ is a graded subspace of $U(\g)$.
It is easy to see from (\ref{stuff}) and (\ref{stuffy}) that this
subspace is the subalgebra $U(\g')$.
In particular, this implies that the ordered monomials that span
$\overline{SY}_n$ are linearly independent too.
Furthermore, multiplying them by ordered monomials in $\big\{D_{1}^{(r)}\:\big|\:r > 0\big\}$ 
gives a basis for $Y_n$.
This shows that the linear map
\begin{align*}
\overline{SY}_n \otimes \k[x_1,x_2,\dots] &\rightarrow Y_n,\\
A(x_1,\dots,x_n) = \sum_i A_i \otimes a_i(x_1,\dots,x_n) &\mapsto
A(D_1^{(1)}, \dots, D_1^{(n)}) = \sum_i A_i a_i\big(D_1^{(1)}, \dots D_1^{(n)}\big)
\end{align*}
is a vector space isomorphism.

Now we can show that
 $SY_n \subseteq
\overline{SY}_n$.
Take any $B \in SY_n$. By the previous paragraph, we can write it as $A\big(D_1^{(1)}, \dots,
D_1^{(n)}\big)$
for a unique  
$A(x_1,\dots,x_n) \in \overline{SY}_n \otimes \k[x_1,\dots,x_n]$
and
$n \geq 1$.
Taking $f(u) := 1 + c u^{-n}$ for $c \in \K$ and an infinite field $\K
\supseteq \k$, we have that 
$$
A\big(D_1^{(1)}, \dots,
D_1^{(n)}\big)\otimes 1=
\mu_f\left(A\big(D_1^{(1)}, \dots,
D_1^{(n)}\big)\otimes 1\right)
= A\big(D_1^{(1)},\dots,D_1^{(n-1)}, D_1^{(n)}+c\big) \otimes 1.
$$
This implies that $A(x_1,\dots,x_n) = A(x_1,\dots,x_{n-1},x_n+c)$
for infinitely many $c$. Hence, $A(x_1,\dots,x_n)$ is independent of
$x_n$.
Similarly, it is independent of $x_{n-1},\dots,x_1$. This shows
$B \in \overline{SY}_n$ as
required.

We have now shown that the ordered monomials in the elements
(\ref{beans}) give a basis for $SY_n$, that $\gr SY_n = U(\g')$, and
that the map (\ref{beans1}) is a vector space isomorphism.
Finally, we must prove (\ref{beans2}) assuming $p \nmid n$.
The vector space $SY_n \otimes Z_\HC(Y_n)$ has a basis given by the
ordered monomials in (\ref{beans}) tensored with
ordered monomials in $\big\{C^{(r)}\:\big|\:r > 0\big\}$.
It suffices to show that the images of these give a basis for $Y_n$.
By passing to the associated graded algebra as
usual, this reduces to the observation that
$$
\big\{e_{i,i}t^r -e_{i+1,i+1}t^r \:\big|\:i=1,\dots,n-1, r \geq 0\big\}\cup
\big\{e_{i,j} t^r\:\big|\:1 \leq i\neq j \leq n, r \geq 0\big\} \cup \{z_r\:|\:r
\geq 0\}
$$
is a basis for $\g$.
\end{proof}

\begin{Remark}
Recall that $\widetilde{T}_{i,j}(u) = S(T_{i,j}(u))$.
Since $\mu_f(\widetilde{T}_{i,j}(u)) = f(u)^{-1}
\widetilde{T}_{i,j}(u)$,
the definition (\ref{poppy}) implies that
all coefficients of $T_{i,j}(u) \widetilde{T}_{k,l}(u)$
belong to $SY_n$.
By passing to $\gr Y_n$ and using Theorem~\ref{oldsy}, one sees that
these coefficients also generate $SY_n$.
Using this and (\ref{comult}), it follows that
$\Delta(SY_n) \subseteq SY_n \otimes SY_n$, so that $SY_n$ is a Hopf
subalgebra of $Y_n$; cf. \cite[Proposition 2.21]{MNO}.
\end{Remark}

\begin{Theorem}
\label{newsy}
The algebra $SY_n$ is generated by the elements 
\begin{equation}\label{beany}
\big\{H_i^{(r)}, E_i^{(r)}, F_i^{(r)}\:\big|\:1\leq i <n, r > 0\big\}
\end{equation}
subject
only to 
the relations (\ref{r6})--(\ref{rel16}) plus the following:
\begin{align}
\big[H_i^{(r)}, H_j^{(s)}\big] &= 0,\label{nr1}\\
\big[E_i^{(r)}, F_j^{(s)}\big] &= \delta_{i,j} H_i^{r+s-1},\label{nr2}
\end{align}\begin{align}
\label{nr3}
\big[H_i^{(r)}, E_j^{(s)}\big] &= 0\hspace{18mm}\text{if $|i-j|>1$},\\
\big[H_i^{(r)}, F_j^{(s)}\big] &= 0\hspace{18mm}\text{if $|i-j|>1$},\\
\label{nr4}
\big[H_{i-1}^{(r+1)}, E_i^{(s)}\big]-\big[H_{i-1}^{(r)}, E_i^{(s+1)}\big]
&= 
H_{i-1}^{(r)} E_i^{(s)},\\\label{nr4b}
\big[H_{i-1}^{(r)}, F_i^{(s+1)}\big]
-\big[H_{i-1}^{(r+1)}, F_i^{(s)}\big]
&= 
F_i^{(s)} H_{i-1}^{(r)},\\\label{nr5}
\big[H_{i}^{(r+1)}, E_i^{(s)}\big]-\big[H_{i}^{(r)}, E_i^{(s+1)}\big]
&= 
-H_{i}^{(r)} E_i^{(s)}-E_i^{(s)} H_{i}^{(r)},\\\label{nr5b}
\big[H_{i}^{(r)}, F_i^{(s+1)}\big]-\big[H_{i}^{(r+1)}, F_i^{(s)}\big]
&= 
-F_i^{(s)} H_{i}^{(r)}-H_{i}^{(r)} F_i^{(s)},\\
\big[H_{i+1}^{(r+1)}, E_i^{(s)}\big]-\big[H_{i+1}^{(r)}, E_i^{(s+1)}\big]
&= 
E_i^{(s)} H_{i+1}^{(r)},\label{nr6}\\
\big[H_{i+1}^{(r)}, F_i^{(s+1)}\big]-\big[H_{i+1}^{(r+1)}, F_i^{(s)}\big]
&= 
H_{i+1}^{(r)}F_i^{(s)},\label{nr6b}
\end{align}
for all admissible $i,j,r,s,t$ including $r=0$ in
(\ref{nr4})--(\ref{nr6b}); remember also $H_i^{(0)} = -1$.
\end{Theorem}

\begin{proof}
In view of (\ref{eij}), Theorem~\ref{oldsy} implies that $SY_n$ is generated by
the elements (\ref{beany}).
Let us also show that all of the relations in the theorem are
satisfied.
Of course (\ref{r6})--(\ref{rel16}) hold, and
the relations (\ref{nr1})--(\ref{nr2}) follow from
(\ref{r2})--(\ref{r3}).
For the remaining relations, (\ref{nr4}), (\ref{nr5}) and (\ref{nr6})
follow by equating 
$u^{-r} v^{-s}$-coefficients
in (\ref{h1}), (\ref{hcor0}) and (\ref{h3}), respectively,
taking $r \geq 0$ and $s > 0$.
Then (\ref{nr4b}), (\ref{nr5b}) and (\ref{nr6b}) follow by applying the anti-automorphism
$\tau$.

Now let $\widehat{SY}_n$ be the algebra defined by the
generators and relations from the theorem.
The previous paragraph shows that there is an algebra
homomorphism
$\widehat{SY}_n \twoheadrightarrow SY_n$ taking generators to
generators.
To show that it is an isomorphism, define elements $E_{i,j}^{(r)},
F_{i,j}^{(r)} \in \widehat{SY}_n$ by (\ref{eij}).
Using the basis
for $SY_n$ constructed in Theorem~\ref{oldsy}, it suffices to show that
ordered monomials in the elements (\ref{beans}) span $\widehat{SY}_n$.
Moreover, we can choose the order so that the elements $F_{i,j}^{(r)}$
come first, followed by the elements $H_i^{(r}$, followed by the
elements $E_{i,j}^{(r)}$.
Of course, $\widehat{SY}_n$ is spanned by unordered monomials in
$F_i^{(r)}, H_i^{(r)}$ and $E_i^{(r)}$.
The relations allow us to inductively commute all $F_i^{(r)}$ to the beginning, 
all $H_i^{(r)}$ to the middle and in the chosen order due to (\ref{nr1}),
and all $E_i^{(r)}$ to the end. Then we get
done because, in the subalgebras generated by just the $F_i^{(r)}$ or
the $E_i^{(r)}$, we have available exactly the same relations as in $Y_n$,
and there we have already established that the given ordered monomials span
these subalgebras
using the relations  (\ref{r6})--(\ref{rel16})
and (\ref{eij}).
\end{proof}

\begin{Remark}\label{dp}
Assume $\operatorname{char}\k \neq 2$.
Then there is an even more efficient presentation for $SY_n$, namely,
the usual Drinfeld presentation for the Yangian of $\sl_n$
from \cite{D} na\"ively taken over the field $\k$ rather than over the
complex numbers; actually, we use the ``opposite'' presentation like
in \cite[Remark 5.12]{BK1}.
In more detail, we define new elements
\begin{align}
\kappa_i(u) = \sum_{k \geq 0} \kappa_{i,k}u^{-k-1}
&:= 1+\eta_{(i-1)/2}(H_i(u)),\\
\xi_i^+(u) = \sum_{k \geq 0} \xi_{i,k}^+ u^{-k-1} &:=
\eta_{(i-1)/2}(E_i(u)),\\
\xi_i^-(u) = \sum_{k \geq 0} \xi_{i,k}^- u^{-k-1} &:= \eta_{(i-1)/2}\left(F_{i}(u)\right),
\end{align}
where $\eta_c$ is the automorphism from \textsection \ref{saut} (which
leaves $SY_n$ invariant).
Then, 
the algebra $SY_n$ is generated by 
$\big\{\kappa_{i,k}, \xi_{i,k}^{\pm}\:|\:1 \leq i < n, k \geq 0 \big\}$
subject only to the Drinfeld relations:
\begin{align}\label{dr1}
\big[\kappa_{i,k}, \kappa_{j,l}\big] &= 0, \\
\big[\xi^+_{i,k}, \xi^-_{j,l}\big] &= \delta_{i,j} \kappa_{i,k+l},\\
\big[\kappa_{i,0}, \xi^{\pm}_{j,l}\big] &= \pm a_{i,j} \xi^{\pm}_{j,l},\label{eg0}\\
\big[\kappa_{i,k},\xi^\pm_{j,l+1}\big] - \big[\kappa_{i,k+1}, \xi^{\pm}_{j,l}\big]
&= \pm \frac{a_{i,j}}{2} (\kappa_{i,k} \xi^{\pm}_{j,l} + \xi^{\pm}_{j,l} \kappa_{i,k}),\label{eg1}\\
\big[\xi^\pm_{i,k}, \xi^\pm_{j,l+1}\big] - \big[\xi^\pm_{i,k+1}, \xi^\pm_{j,l}\big]
&= \pm \frac{a_{i,j}}{2} (\xi^\pm_{i,k} \xi^{\pm}_{j,l} +
\xi^\pm_{j,l} \xi^\pm_{i,k}),\label{eg}\\
\Big[\xi_{i,k_1}^{\pm}, \big[\xi_{i,k_2}^{\pm}, \xi_{j,l}^\pm\big]\Big]
+
\Big[\xi_{i,k_2}^{\pm}, \big[\xi_{i,k_1}^{\pm}, \xi_{j,l}^\pm\big]\Big]
&= 0\text{ if $|i-j|=1$, $k_1 \neq k_2$,}\\
\Big[\xi_{i,k}^{\pm}, \big[\xi_{i,k}^{\pm}, \xi_{j,l}^\pm\big]\Big] &=
0\text{ if $|i-j|=1$,}\\
\big[\xi_{i,k}^{\pm}, \xi_{j,l}^\pm\big] &= 0\text{ if $|i-j|>1$,}
\label{drn}
\end{align}
for $a_{i,j} := 2 \delta_{i,j} - \delta_{i,j+1}-\delta_{i,j-1}$
(the Cartan matrix of type $A_{n-1}$).
This assertion is just a rephrasing of Theorem~\ref{newsy} for these modified
generators. For example, the relation (\ref{eg}) is deduced 
in \cite[Remark 5.12]{BK1};
the relations (\ref{eg0})--(\ref{eg1}) follow from
(\ref{hcor1})--(\ref{hcor3}) suitably shifted.
\end{Remark}

\begin{Remark}
There is a shifted analogue $SY_n(\sigma)$ of $SY_n$. This may be
realized as a subalgebra of $Y_n(\sigma)$ via a similar definition to
(\ref{poppy}).
The presentation in Theorem~\ref{newsy} can also be modified to give a
presentation for $SY_n(\sigma)$, in just the same way that
Theorem~\ref{drinfeldshift}
modifies Theorem~\ref{drinpres}; we leave the details of this to the reader.
The Drinfeld presentation from the previous remark
does not immediately 
make sense for $SY_n(\sigma)$, but see \cite{WWY} for a closely related
result (in characteristic zero).
\end{Remark}

\subsection{\boldmath The $p$-centre of $SY_n$}
Let
\begin{multline}
\qquad A_i(u) = \sum_{r \geq 0} A_i^{(r)} u^{-r} := H_i(u) H_i(u-1)\cdots
H_i(u-p+1)\\ = -B_{i+1}(u)B_i(u)^{-1}
\in SY_n[[u^{-1}]].\qquad
\end{multline}
In view of Theorem~\ref{dc},
each $A_i^{(r)}$ belongs to $Z(SY_n)$.
We define the {\em $p$-centre} of $SY_n$ to be the subalgebra
$Z_p(SY_n)$ of
$Z(SY_n)$ generated by
\begin{equation}\label{xz}
\big\{A_i^{(rp)}\:\big|\:1 \leq i < n, r > 0\big\}\cup
\Big\{\big(E_{i,j}^{(r)} \big)^p,
\big(F_{i,j}^{(r)} \big)^p\:\Big|\:1 \leq i< j \leq n, r > 0\Big\}.
\end{equation}
Also let $Z_p(\g')$ be the $p$-centre of $U(\g')$, i.e. the subalgebra of $Z(\g')$ 
generated by $x^p - x^{[p]}$ for all $x \in \g'$.

\begin{Theorem}
The generators (\ref{xz}) of $Z_p(SY_n)$ are algebraically
independent,
and we have that $\gr Z_p(SY_n) = Z_p(\g')$.
Moreover, $SY_n$
is free as a module over $Z_p(SY_n)$ with
basis given by the ordered monomials in 
\begin{equation}\label{mon3}
\big\{H_i^{(r)}\:\big|\:1 \leq i < n, r >
0\big\}\cup\big\{E_{i,j}^{(r)},
F_{i,j}^{(r)}\:\big|\:1\leq i < j \leq n, r > 0\big\}
\end{equation}
in which no exponent is $p$ or more.
\end{Theorem}

\begin{proof}
From the formula
$A_i(u) = -B_{i+1}(u)B_i(u)^{-1}$, we get that
$A_i^{(r)} = B_i^{(r)} - B_{i+1}^{(r)}$ plus a linear combination of
monomials $B_i^{(r_1)}\cdots B_i^{(r_k)}
B_{i+1}^{(s_1)}\cdots B_{i+1}^{(s_l)}$
with $r_1+\cdots+r_k+s_1+\cdots+s_l = r$.
Combined with (\ref{creek}), it follows that $A_i^{(rp)} \in \F_{rp-p}
Y_n$ and
\begin{equation}\label{lastylifty}
\gr_{rp-p} A_i^{(rp)} = \big(e_{i,i} t^{r-1}
-e_{i+1,i+1} t^{r-1}\big)^p - \big(e_{i,i} t^{rp-p}- e_{i+1,i+1} t^{rp-p}\big).
\end{equation}
Then from (\ref{lastylifty}) and (\ref{snore}), we see
that the generators (\ref{xz}) of $Z_p(SY_n)$
are lifts of generators for $Z_p(\g')$ coming from a basis for $\g'$.
This establishes the algebraic independence and that $\gr Z_p(SY_n) =
Z_p(\g')$.
The final part of the theorem 
follows by similar argument to the proof of Corollary~\ref{nearly2},
using the PBW basis for $SY_n$ from Theorem~\ref{oldsy}.
\end{proof}

In fact, when $p \nmid n$, the $p$-center of $SY_n$ is the full
center, thanks to the following theorem.
This is the positive characteristic counterpart of the observation that $Z(SY_n)$ is
trivial in characteristic zero from \cite[Proposition 2.16]{MNO}.

\begin{Theorem}\label{reallytheend}
If $p \nmid n$ then $Z_p(SY_n) = Z(SY_n)$.
\end{Theorem}

\begin{proof}
We have that $Z_p(SY_n) \subseteq Z(SY_n)$, hence,
$\gr Z_p(SY_n) \subseteq \gr Z(SY_n) \subseteq Z(\g')$.
In the next paragraph, we show that $Z_p(\g') = Z(\g')$.
We also know that $\gr Z_p(SY_n) = Z_p(\g')$ from the previous
theorem.
Then we get that
$\gr Z_p(SY_n) = \gr Z(SY_n)$
implying that $Z_p(SY_n) = Z(SY_n)$.

To show that $Z_p(\g') = Z(\g')$, the assumption $p\nmid n$ implies
that $\g = \g' \oplus \mathfrak{z}(\g)$.
Hence, $Z(\g) \cong Z(\g') \otimes \k[z_r\:|\:r \geq 0]$.
It remains to observe that the elements
$\{x^p - x^{[p]}\:|\:x \in \g'\} \cup \{z_r\:|\:r \geq 0\}$
generate $Z(\g)$.
This follows from
Theorem~\ref{loopcentre} using the assumption $p \nmid n$.
\end{proof}

\subsection{\boldmath Another description of the $p$-centre of $Y_n$}
Recall by the definition (\ref{pcent}) and Theorems~\ref{Z1} and \ref{dc} 
that the $p$-centre $Z_p(Y_n)$ is the subalgebra of $Z(Y_n)$ generated
by the coefficients of the power series 
$B_i(u)$, $P_{i,j}(u)$ and $Q_{i,j}(u)$.
Let
\begin{equation}\label{protest}
S_{i,j}(u) = \sum_{r \geq 0} S_{i,j}^{(r)} u^{-r} 
:= T_{i,j}(u) T_{i,j}(u-1)\cdots T_{i,j}(u-p+1)
\in Y_n[[u^{-1}]].
\end{equation}
In view of (\ref{Theycommute}), the order of the product on the right hand side
here is irrelevant.

\begin{Lemma}\label{lastlem}
All of the elements $S_{i,j}^{(r)}$ belong to the $p$-centre
$Z_p(Y_n)$.
\end{Lemma}

\begin{proof}
First we show that each $S_{i,j}^{(r)}$ belongs to $Z(Y_n)$.
To see this, using the
conjugation automorphism
from \textsection \ref{saut} which sends $S_{i,j}(u)$ to $S_{w(i),
  w(j)}(u)$,
we reduce to proving that all coefficients of $S_{1,1}(u)$ and of
$S_{1,2}(u)$ are central.
The latter assertions follow because
\begin{align}
S_{1,1}(u) &
= B_1(u),\label{ob}\\
S_{1,2}(u) &= B_1(u) P_{1,2}(u).\label{od}
\end{align}
The first identity (\ref{ob}) here is immediate as $T_{1,1}(u) = D_1(u)$.
To prove (\ref{od}), we set $v = u-m$ in (\ref{Beq3}) to deduce 
that
\begin{equation*}
E_i(u-m) D_i(u-m+1)\cdots D_i(u-1) D_i(u)
 =D_i(u-m+1)\cdots D_i(u-1) D_i(u) E_i(u)
\end{equation*}
for each $m=1,\dots,p-1$.
The $(1,2)$-entry of (\ref{gfact}) gives that $T_{1,2}(u) = D_1(u)
E_1(u)$.
Hence, we get that
\begin{align*}
S_{1,2}(u) &= T_{1,2}(u-p+1)\cdots T_{1,2}(u-1) T_{1,2}(u)\\
&=
D_1(u-p+1) E_1(u-p+1) \cdots D_1(u-1) E_1(u-1) D_1(u) E_1(u)
\\&=
D_1(u-p+1) \cdots D_1(u-1) D_1(u) E_1(u)^p = B_1(u) P_{1,2}(u).
\end{align*}
This establishes (\ref{od}).

The delicate point now is to show that $S_{i,j}^{(r)}$ actually lies in
$Z_p(Y_n)$ not just $Z(Y_n)$.
By Theorem~\ref{main}(2), we have that 
$Z_p(Y_n) = Y_n \cap
Z_p(Y_{n+1})$,
where we are implicitly using the natural embedding $Y_n\hookrightarrow
Y_{n+1}, T_{i,j}^{(r)} \mapsto T_{i,j}^{(r)}$.
Hence,
in order to prove that $S_{i,j}^{(r)} \in Z_p(Y_n)$, we
may assume
that $p \nmid n$.

So finally we assume $p \nmid n$ and show that $S_{i,j}(u) \in
Z_p(Y_n)[[u^{-1}]]$.
This is immediate by (\ref{ob}) in case $i=j=1$.
In general, we show equivalently that
$S_{i,j}(u) S_{1,1}(u)^{-1} \in Z_p(Y_n)[[u^{-1}]]$.
Using the definition (\ref{poppy}), we get that $S_{i,j}(u)
S_{1,1}(u)^{-1} \in SY_n[[u^{-1}]]$. Since we have shown its
coefficients are central already, it therefore lies in 
$Z(SY_n)[[u^{-1}]]$, which 
by Theorem~\ref{reallytheend} and the definitions is $Z_p(SY_n)[[u^{-1}]]\subset
Z_p(Y_n)[[u^{-1}]]$.
\end{proof}

\begin{Theorem}\label{finalt}
The $p$-centre $Z_p(Y_n)$ is freely generated by 
$\big\{S_{i,j}^{(rp)}\:\big|\:1 \leq i,j \leq n, r > 0\big\}$.
We have that $S_{i,j}^{(rp)} \in \F_{rp-p} Y_n$ and
\begin{equation}\label{this}
\gr_{rp-p} S_{i,j}^{(rp)} = \big(e_{i,j} t^{r-1}\big)^p
- \delta_{i,j} e_{i,j} t^{rp-p}.
\end{equation}
For $0 < r< p$, we have that $S_{i,j}^{(r)} = 0$.
For $r \geq p$ with $p \nmid r$, the central element $S_{i,j}^{(r)}$ 
belongs to $\F_{rp-p-1} Y_n$ and it may be expressed as a polynomial
in the elements $\big\{S_{i,j}^{(ps)}\:\big|\:0 < s \leq \lfloor r/p
\rfloor \big\}$.
\end{Theorem}

\begin{proof}
When $n=1$, the first statement follows immediately from the definition of
$Z_p(Y_n)$, remembering (\ref{ob}). The remaining statements follow too if we can prove them for larger $n$.
So we assume from now on that $n \geq 2$.
To prove (\ref{this}), 
we apply Lemma~\ref{birds} if $i=j$ or Lemma~\ref{lemIV} if $i \neq j$,
taking $X^{(r)} := T_{i,j}^{(r)}$.
These lemmas also show that
$S_{i,j}^{(r)} = 0$ for $0 < r < p$ and that
$S_{i,j}^{(r)} \in \F_{rp-p-1} Y_n$ when $p \nmid r \geq p$.

Let $Y_{i,j}$ be the subalgebra of $Y_n$ generated by the elements
$\big\{T_{i,j}^{(r)}\:\big|\:r > 0\big\}$ and $\g_{i,j}$ be the
subalgebra of $\g$ spanned by $\{e_{i,j}t^r\:|\:r \geq 0\}$.
We have that $\gr Y_{i,j} = U(\g_{i,j})$ and
\begin{equation}
Z(\g) \cap U(\g_{i,j}) = \k\left[\big(e_{i,j}t^r\big)^p - \delta_{i,j}
  e_{i,j} t^{rp}\:\Big|\:r \geq 0\right]
\end{equation}
just like in (\ref{finalz1}) and (\ref{finalz2}).
Combined with (\ref{this}),
it follows that $Z(Y_n) \cap Y_{i,j}$ is freely generated by
$\big\{S_{i,j}^{(rp)}\:\big|\:r > 0\big\}$; this is exactly the same
argument as used in the proofs of Theorems~\ref{Z1} and \ref{dc}.
The last assertion in the statement of the theorem follows.

Finally, we must prove the first assertion.
Lemma~\ref{lastlem} shows that $S_{i,j}^{(rp)}$ lies in $Z_p(Y_n)$.
To show that 
$\big\{S_{i,j}^{(rp)}\:\big|\:1 \leq i,j \leq n, r > 0\big\}$
are algebraically independent and generate
$Z_p(Y_n)$, we pass to the associated graded algebra
using (\ref{this}),
to see that they are lifts of the generators of
$\gr Z_p(Y_n) = Z_p(\g)$ from (\ref{zdf}).
\end{proof}

\begin{Remark}\label{rainer}
Similar to Remark~\ref{later2}, Theorem~\ref{finalt} justifies Remark~\ref{rainier}.
\end{Remark}

When combined with Theorem~\ref{main} and Corollary~\ref{nearly}, 
Theorem~\ref{finalt} finally establishes all of
the statements about $Z(Y_n)$ that we formulated in the introduction.
We should also note for this that the central elements
$C^{(r)}$ defined by the quantum determinant in the introduction 
are the same as the ones 
arising from (\ref{Cdef}).
This is a non-trivial observation which is proved in characteristic
zero in \cite[Theorem 8.6]{BK1};  the argument there works over $\Z$, hence,
also in positive characteristic.

\end{document}